%% file: LiGoSPAarXiv.tex
\documentclass[%
  onecolumn
   , hidempi
]{mpi2015-cscpreprint}

\usepackage{accents}
\usepackage[toc,page]{appendix}
\usepackage{enumerate}
\usepackage{amssymb, amsthm}

\newtheorem{definition}{Definition}[section]
\newtheorem{proposition}{Proposition}[section]
\newtheorem{theorem}{Theorem}[section]

\newtheorem{remark}{Remark}[section]

\usepackage{multirow}

\usepackage{verbatim}
\usepackage{xspace,extarrows}
\usepackage[colorinlistoftodos]{todonotes}
\usepackage{caption}
\usepackage{subcaption}
\usepackage{sidecap}

\usepackage{algorithmic}
\usepackage{etoolbox}
\AtBeginEnvironment{algorithmic}{\footnotesize}
\usepackage{algorithm}

\input{mymacros}

\usepackage{amsmath,amsfonts,amssymb,amsthm,epsfig,graphicx,url}
\usepackage{xspace}
\usepackage{mathrsfs}
\usepackage{cite}

\usepackage{tikz,tikz-network}
\usetikzlibrary{calc,positioning,shapes,arrows,decorations.shapes} 
\tikzstyle{arr}=[-latex, black, line width=0.5pt]
\tikzstyle{doublearr}=[latex-latex, black, line width=0.5pt]
\tikzstyle{input}=[font=\small\sffamily\bfseries]
\tikzstyle{rect}=[
   rectangle, draw=black, font=\large\sffamily\bfseries, inner sep=11pt]
\tikzstyle{rect2}=[
   rectangle, draw=black, font=\large\sffamily\bfseries, inner sep=4pt]
\usepackage{caption}

\author[$\ast$]{Bj\"orn Liljegren-Sailer}
\affil[$\ast$]{Trier University,\authorcr
	Universit\"atsring 15, 54296 Trier, Germany.\authorcr
  \email{bjoern.sailer@uni-trier.de}, \orcid{0000-0002-5267-7801}}

\author[$\dagger$]{Ion Victor Gosea}
\affil[$\dagger$]{Max Planck Institute for Dynamics of Complex Technical Systems,\authorcr
	Sandtorstr. 1, 39106 Magdeburg, Germany.\authorcr
  \email{gosea@mpi-magdeburg.mpg.de}, \orcid{0000-0003-3580-4116}}

\title{Data-driven and low-rank implementations of Balanced Singular Perturbation Approximation}  
\shorttitle{Data-driven and low-rank implementations of SPA}
\shortauthor{B. Liljegren-Sailer, I. V. Gosea}
\shortdate{}


\abstract{
Balanced Singular Perturbation Approximation (\SPA) is a model order reduction method for linear time-invariant systems that guarantees asymptotic stability and for which there exists an a priori error bound. In that respect, it is similar to Balanced Truncation (\BT). However, the reduced models obtained by \SPA generally introduce better approximation in the lower frequency range and near steady-states, whereas \BT is better suited for the higher frequency range. Even so, independently of the frequency range of interest, \BT and its variants are more often applied in practice, since there exist more efficient algorithmic realizations thereof.

In this paper, we aim at closing this practically-relevant gap for \SPA. We propose two novel and efficient algorithms that are adapted for different settings. Firstly, we derive a low-rank implementation of \SPA that is applicable in the large-scale setting. Secondly, a data-driven reinterpretation of the method is proposed that only requires input-output data, and thus, is realization-free.
A main tool for our derivations is the reciprocal transformation, which induces a distinct view on implementing the method. While the reciprocal transformation and the characterization of \SPA is not new, its significance for the practical algorithmic realization has been overlooked in the literature. Our proposed algorithms have well-established counterparts for \BT, and as such, also a comparable computational complexity. The numerical performance of the two novel implementations is tested for several numerical benchmarks, and comparisons to their counterparts for \BT as well as the existing implementations of \SPA are made.
}
\begin{document}

\maketitle

\date{\today}


\input{secIntroBasics}
\input{secSPA}

\input{secLowRank}


\input{secQuad}


\input{secNumQuad}


\input{secConc}

\appendix

\input{AppQuad}

\bibliographystyle{plainurl}
\bibliography{LiGoSPA}   

\end{document}

%% file: mymacros.tex
\newcommand{\re}[1]{\mathfrak{{#1}}} 
\newcommand{\rer}[1]{\widetilde{\mathfrak{{#1}}}_r} 
\newcommand{\qdr}[1]{\tilde{{#1}}} 

\newcommand{\bv}[1]{\mathbf{{#1}}} 

\newcommand{\bal}[1]{\hat{#1}} 
\def\balT{\bv{F}} 

\def\TF{\mathbf{H}} 
\def\TFre{{\mathfrak{H}}} 
\def\spTF{{\mathbf{H}}_\infty} 

\def\TFzer{{\mathbf{H}}_0} 

\def\Kr{{\mathcal{K}}}

\def\Kre{{\mathfrak{K}}} 
\def\FF{{\re{\bv{F}}}} 

\def\qwe{{\varphi }} 
\def\qpo{{\zeta}} 

\def\IR{{\mathbb R}}
\def\IC{{\mathbb C}}

\newcommand{\bA}{{\textbf A}}
\newcommand{\bB}{{\textbf B}}
\newcommand{\bC}{{\textbf C}}
\newcommand{\bD}{{\textbf D}}
\newcommand{\bE}{{\textbf E}}

\newcommand{\bS}{{\textbf S}}
\newcommand{\bY}{{\textbf Y}}

\newcommand{\bL}{{\textbf L}}

\newcommand{\bI}{{\textbf I}}
\newcommand{\bH}{{\textbf H}}
\newcommand{\bP}{{\textbf P}}

\newcommand{\bQ}{{\textbf Q}}
\newcommand{\bW}{{\textbf W}}

\newcommand{\bZ}{{\textbf Z}}
\newcommand{\bX}{{\textbf X}}
\newcommand{\bx}{{\textbf x}}
\newcommand{\by}{{\textbf y}}
\newcommand{\bu}{{\textbf u}}
\newcommand{\bV}{{\textbf V}}
\newcommand{\bU}{{\textbf U}}

\newcommand{\cH}{ {\cal H} }

\newcommand{\np}{{{}{J}}} 
\newcommand{\nq}{{{}N_q}} 

\newcommand{\quadP}{\widetilde{\bP}}
\newcommand{\quadQ}{\widetilde{\bQ}}
\newcommand{\quadU}{\widetilde{\bU}}
\newcommand{\quadL}{\widetilde{\bL}}
\newcommand{\quadZ}{\widetilde{\bZ}}
\newcommand{\quadS}{\widetilde{\bS}}
\newcommand{\quadY}{\widetilde{\bY}}

\newcommand{\quadLb}{\widetilde{\bold{T}}}
\newcommand{\quadcU}{\widetilde{\bv{G}}}

\newcommand{\quadLL}{\widetilde{\bv{N}}}
\newcommand{\quadMM}{\widetilde{\bv{M}}}

\newcommand{\requadLb}{\widetilde{\re{T}}}
\newcommand{\requadcU}{\widetilde{\re{G}}}

\newcommand{\requadLL}{\widetilde{\re{N}}}
\newcommand{\requadMM}{\widetilde{\re{M}}}
\newcommand{\requadU}{\widetilde{\re{U}}}
\newcommand{\requadL}{\widetilde{\re{L}}}
\newcommand{\requadP}{\widetilde{\re{P}}}
\newcommand{\requadQ}{\widetilde{\re{Q}}}
\newcommand{\ROM}{\textsf{ROM}\xspace}
\newcommand{\FOM}{\textsf{FOM}\xspace}
\newcommand{\LTI}{\textsf{LTI}\xspace}
\newcommand{\QBT}{\textsf{QuadBT}\xspace}
\newcommand{\BT}{\textsf{BT}\xspace}
\newcommand{\SPA}{\textsf{SPA}\xspace}
\newcommand{\QSPA}{\textsf{QuadSPA}\xspace}
\newcommand{\imunit}{{\dot{\imath\hspace*{-0.2em}\imath}}}

\newcommand{\build}{\textsf{LAbuild}\xspace}
\newcommand{\iss}{\textsf{ISS12A}\xspace}

\newcommand{\rail}{\textsf{Rail}\xspace}



%% file: secIntroBasics.tex
\section{Introduction} 
\label{intro}

For many modern applications in the applied sciences, the simulation, control and optimization of very large-scale dynamical systems is required.  The increased development of modern computing environments and high-performance computing tools has pushed the boundaries of what is computationally feasible in this regard. Another way to go about this issue and to save computational time is by my means of approximating such large-scale dynamical systems with reduced models. This is the essence of model reduction, which is a sub-field at the intersection of many established fields, such as automatic control, systems theory, approximation theory, as well as numerical linear algebra; see e.g., the following books \cite{ACA05,quarteroni2015reduced,BOCW17,HandbookVol1,AntBG20}.

For the reduction of linear time-invariant systems, system-theo{\-}re{\-}tic methods are widely used. Moment matching methods and projection-based balancing-related methods, such as Balanced Truncation (\BT), are among the most popular methods falling in this category, see \cite{breiten20212,benner2021model} and references therein. Moment matching methods are particularly efficient in terms of computational complexity. The resulting reduced models are also typically of high fidelity, but it is in general difficult to enforce theoretical guarantees.

The balancing-related methods preserve important qualitative features of a system, such as stability, and there exists a priori error bounds for those methods \cite{glover1984all}. The concept behind \BT, originally proposed in \cite{mullis1976synthesis,moore1981principal}, is to identify and truncate components of the original system that are weakly controllable and observable. This turns out to be a natural approach assuming the truncated states have a fast dynamic.
Another closely related method is the (Balanced) Singular Perturbation Approximation (\SPA) \cite{fernando1982singular,liu1989singular}. It is a non-projection based balancing-related method, for which the same a priori guarantees as for \BT are valid. Other than that, \SPA is better suited for the lower frequency range and near steady states.

Notably, the numerical algorithms available for \BT are more developed than those for \SPA. For example, low-rank implementations are exclusive to \BT and thus \SPA, until now, could not be applied in the truly large-scale setting; cf.\,\Cref{sec:LowRnkImp} for references and more details.
The recent contribution \cite{gosea2022data} connects classical projection-based \BT approach to data-driven interpolation-based approach and provides a non-intrusive implementation of the method.
With the ever-increasing availability of data, i.e., measurements related to the original model, the interest in non-intrusive methods has grown significantly. Among other methods, we would like to mention the Loewner framework  \cite{ajm07}, Dynamic Mode Decomposition \cite{schmid_2010} and Operator Inference \cite{PEHERSTORFER2016196,morBenGKetal20}.

The main objective of this paper is twofold: first, we aim at deriving an efficient algorithm for \SPA in a large-scale setting and, second, at providing a realization-free implementation of \SPA. 
The latter will be performed in a non-intrusive manner, using solely transfer function evaluations and corresponding quadrature weights.
Both new implementations have strong resemblances with existing implementations of \BT. The essential link between \BT and \SPA is given by the so-called reciprocal transformation, and this plays a crucial role in our derivations.

The remainder of the paper is organized as follows:~\Cref{sec:BT} provides an overview on balancing and the balancing-related model reduction methods \BT and \SPA and points towards the open issues for the implementation of \SPA.

Then, in \Cref{sec:recTrafo}, the reciprocal transformation is defined and analyzed in a slightly generalized setting as compared to the literature. In \Cref{sec:LowRnkImp} we recall the square-root and low-rank implementations of \BT and then, present a novel counterpart for \SPA. Numerical results for the novel low-rank algorithm are provided. Next, \Cref{sec:QuadBT} introduces in detail the data-driven interpretation of \BT, and, by that, sets the stage for the newly-proposed data-driven implementation of \SPA. Moreover, connections to the Loewner framework are discussed. An extensive numerical study for the data-driven methods is provided in \Cref{sec:Numerics2} for two benchmark examples. Finally, in \Cref{sec:Conclusion}, the conclusions of the paper are outlined, together with a brief outlook to future research endeavors.

\section{Balancing and balancing-related model reduction} \label{sec:BT}

Consider the linear time-invariant (\LTI) dynamical system
\begin{align}
\begin{split} \label{DesSys}
\bE \, \dot{\bx}(t) &= \bA \, \bx(t) + \bB \, \bu(t), \\
\by(t) &= \bC \, \bx(t) + \bD \,\bu(t),
\end{split}
\end{align}
where the input mapping is given by $\bu : [0,\infty] \rightarrow \mathbb{R}^m$, the (generalized) state variable is
$\bx: [0,\infty] \rightarrow \mathbb{R}^n$, and the output mapping is $\by : [0,\infty] \rightarrow \IR^p$. The system matrices are given by $\bE, \bA \in \mathbb{R}^{n \times n}$, and  $\bB\in \mathbb{R}^{n\times m}$, $\bC \in \mathbb{R}^{p\times n}$, $\bD \in \IR^{p \times m}$. The system is closed by choosing initial conditions $\bx(0) = \bv{x}^0 \in \IR^n$.
We assume the system to be asymptotically stable, i.e., $\bE$ and $\bA$ are nonsingular, and the eigenvalues of $\bE^{-1} \bA$ are located in the open left half-plane. In general, the state dimension $n$ may be large, and \eqref{DesSys} will also be referred to as full order model (\FOM).

The input-output behavior of the system, for $\bx(0) = \bv{0}$, and in the frequency domain, is fully characterized by the transfer function, given by
\begin{equation}\label{eq:TF_def}
\TF(s)=\bC (s \bE -\bA)^{-1} \bB + \bD, \qquad  \TF(s) \in \IC^{p\times q},
\end{equation}
for any scalar (frequency) $s \in \IC$.

Model reduction aims at computing a \ROM, i.e., a \LTI system of a reduced state dimension $r\ll n$, that is supposed to reproduce a similar output $\by_r \approx \by$ for the same inputs. This \ROM is given by
\begin{align}
\begin{split} \label{ROMSys}
\dot{\bx}_r(t) &= \bA_r\, \bx_r(t) + \bB_r\, \bu(t), \\
\by_r(t) &= \bC_r \, \bx_r(t) + \bD_r \, \bu(t),
\end{split}
\end{align}
with initial conditions $\bx_r(0) = \bx_r^0 \in \IR^r$, and reduced dimension matrices $\bA_r \in \IR^{r \times r}$, $\bB_r \in \IR^{r \times m}$, $\bC_r \in \IR^{p \times r}$ and $\bD_r \in \IR^{p \times m}$. Arguably, the projection-based approach is most commonly used \cite{ACA05,benner2015survey}, in which reduction bases $\bW_r, \bV_r \in \IR^{n \times r}$ with $\bW_r^T \bE \bV_r = \bI_r \in \IR^{r \times r}$ (unit matrix)  are used to construct the reduced model. In that case $\bA_r = \bW_r^T \bA \bV_r$, $\bB_r = \bW_r^T \bA$, $\bC_r = \bC \bV_r$ and $\bD = \bD_r$.

\subsection{Balanced realization}

The two fundamental quantities in \BT and \SPA are the controllability  and observability Gramians $\bP$ and $\bQ$, given as the solutions of the two Lyapunov equations
\begin{align}  \label{lyapFacDes}
\bA \bP \bE^T + \bE \bP \bA^T + \bB\bB^T = \bv{0}, \hspace{1cm}
\bA^T \bQ \bE + \bE^T \bQ \bA + \bC^T \bC = \bv{0}. 
\end{align}
The Gramian $\bP$ induces a measure for the controllability of a state, and likewise, $\bQ$ a measure for the observability. However, these measures are not invariant under state transformations.
For, let a matrix $\balT \in \IR^{n\times n}$ be given with $\balT \bE \balT^{-1} = \bI$, and define the transformed matrices $\bal{\bA} = \balT \bA \balT^{-1}$, $\bal{\bB} = \balT \bB$, and $\bal{\bC} = \bC \balT^{-1}$.
Then the input-output behavior of \eqref{DesSys} is equivalently described by the alternative \LTI realization
\begin{align*}
\begin{split}  
\dot{\bal{\bx}}(t) &= \bal{\bA}\, \bal{\bx}(t) + \bal{\bB}\, \bu(t), \hspace{1.2cm}
\by(t) = \bal{\bC} \, \bal{\bx}(t) + \bD \, \bu(t), \vspace{0.5cm}
\end{split}
\end{align*}
and $\bal{\bv{x}}(0) =  \balT \bv x_0$. A realization is called \emph{balanced}, when its Gramians are diagonal and equal to each other. More specifically, we require for the transformed Gramians $\bal{\bP}$ and $\bal{\bQ}$ that it holds
\begin{align*}
  \bal{\bP} =  \balT \bP \balT^T \stackrel{!}{=}  \bal{\bQ} = \balT^{-T} \bQ \balT^{-1} 
 \stackrel{!}{=} \text{diag}(\sigma_1,\ldots,\sigma_n)  
 \in \IR^{n\times n},
\end{align*}
with $\sigma_1 \geq \sigma_2 \geq \ldots \geq \sigma_n \geq 0$ sorted in a descending order. The $\sigma_i$'s are invariants of the \LTI system (they are the same for any realization of the state-space), and referred to as the Hankel singular values.

\subsection{Balancing-related model reduction}

Balancing-related model reduction methods construct a reduced model based on the most observable and most controllable states, which relate to the dominant few Hankel singular values. Let $1<r<n$ and $\sigma_r > \sigma_{r+1}$ and the balanced state $\bal{\bx} =  \balT \bx$ be partitioned as $\bal{\bx} = [\bal{\bx}_1^T,\bal{\bx}_2^T]^T$,
 with $\bal{\bx}_1: [0,\infty] \to \IR^r$ and the remainder $\bal{\bx}_2: [0,\infty] \to \IR^{n-r}$.
The transformed system matrices are partitioned accordingly,
\begin{align*}
\bal{\bA} &=
 \begin{bmatrix}
\bal{\bA}_{11} & \bal{\bA}_{12} \\
\bal{\bA}_{21} & \bal{\bA}_{22}
\end{bmatrix},
\qquad 
\bal{\bB} = 
 \begin{bmatrix}
\bal{\bB}_{1} \\
\bal{\bB}_{2} 
\end{bmatrix}, \qquad
\bal{\bC} 
= 
\begin{bmatrix}
\bal{\bC}_{1} & \bal{\bC}_{2} 
\end{bmatrix},
\end{align*}
and thus, the following relations hold
\begin{subequations}\label{eq:balSys}
\begin{align}
 \dot{\bal{\bx}}_{1}(t) &= \bal{\bA}_{11} \bal{\bx}_{1}(t) + \bal{\bA}_{12} \bal{\bx}_2(t) + \bal{\bB}_{1} \bu(t)  \label{eq:balSys-a},\\
  \dot{\bal{\bx}}_{2}(t) &= \bal{\bA}_{21} \bal{\bx}_1(t) + \bal{\bA}_{22} \bal{\bx}_2(t) + \bal{\bB}_{2} \bu(t) \label{eq:balSys-b},\\
  \by(t) &= \bal{\bC}_{1} \bal{\bx}_{1}(t)+  \bal{\bC}_{2} \bal{\bx}_{2}(t) +\bD \bu(t). \label{eq:balSys-c}
\end{align}
\end{subequations}
For both \BT and \SPA, reduced models are constructed  by a simplification of \eqref{eq:balSys-a} which can be expressed in terms of the variable $\bx_{r} \approx \bal{\bx}_1$.

\paragraph{Balanced truncation}

In \BT, the reduced model is obtained by completely eliminating the second state, i.e., replacing \eqref{eq:balSys-b} with the simplified dynamics $\bal{\bx}_2 \approx \bv{0}$. The ensuing \ROM is given as in \eqref{ROMSys}, with the following choice of reduced matrices
\begin{align*}
 \bA_r = \bal{\bA}_{11} , \hspace{1cm} 
  \bB_r = \bal{\bB}_{1} , \hspace{1cm}
  \bC_r = \bal{\bC}_{1} , \hspace{1cm}
  \bD_r = \bD.
\end{align*}
This approximation assumes $\bal{\bx}_2$ consists of fast dynamics that decay rapidly to zero.

\paragraph{Singular perturbation approximation}

\SPA pursues, in some sense, a converse path from that of \BT. The remainder state is approximated using the steady-state equation $\frac{d}{dt}\bal{\bx}_2(t)\approx \bv{0}$ for $t\geq 0$. Thus, equation \eqref{eq:balSys-b}  is modified to
$
\bv 0 \approx \bal{\bA}_{21} \bal{\bx}_1(t) + \bal{\bA}_{22} \bal{\bx}_2(t) + \bal{\bB}_{2} \bu(t)
$,
which implies the approximation $\bal{\bv{x}}_2(t)\approx -\bal{\bA}_{22}^{-1} (\bal{\bA}_{21} \bal{\bx}_1(t) + \bal{\bB}_{2} \bu(t))$. A reduced formulation with $\bx_r(t) \approx \bal{\bx}_1(t)$ enforcing this approximated dynamics can be obtained by basic manipulations. This yields the matrices of the \SPA reduced system, which read
\begin{align}
\begin{aligned} \label{eq:spa-matrices}
 \bA_r &= \bal{\bA}_{11} - \bal{\bA}_{12} \bal{\bA}_{22}^{-1} \bal{\bA}_{21} , \hspace{1cm} 
  & \bB_r = \bal{\bB}_{1} - \bal{\bA}_{12} \bal{\bA}_{22}^{-1} \bal{\bB}_{2} , \\
 \bC_r &= \bal{\bC}_{1} - \bal{\bC}_{2} \bal{\bA}_{22}^{-1}  , \hspace{1cm} 
  &\bD_r = \bD - \bal{\bC}_{2} \bal{\bA}_{22}^{-1} \bal{\bB}_{2}. \hspace{0.3cm} 
\end{aligned}
\end{align}
Note that this method leads, in general, to a feedthrough term $\bD_r \neq \bv{0}$, even if $\bD $ was equal to zero. Thus, the \SPA method is not a projection-based approach.

\subsection{Open issues for the implementation of Singular Perturbation Approximation} \label{subsec:efficient-bal-related}
Unfortunately, the numerical determination of a balanced realization is computationally very demanding and not well-conditioned for large-scale systems. 
Therefore, balancing-free implementations have been proposed for various balancing-related model reduction methods. 
Efficient implementations for dense systems of medium size have been derived based on the sign function method and similar techniques, both, for \BT (see \cite{morBenQQ00,book:dimred2003, ACA05}) as well as for \SPA \cite{Varga91,benner2000singular, baur2008gramian}.

Other significant results and application areas have been exclusively developed for \BT and other pro{\-}jection-based balancing-related methods. It is to be noted that the existing implementations of \SPA explicitly use formula \eqref{eq:spa-matrices}, which inherently prohibits the use of low-rank approximations. Furthermore, there has not been proposed a data-driven (realization-free) implementation of \SPA yet in literature. In this paper, we aim at filling this practically-relevant gap for \SPA by deriving low-rank and realization-free algorithms.

%% file: secSPA.tex
\section{Reciprocal transformation} \label{sec:recTrafo}

It is well-known that the \SPA can be interpreted in terms of the so-called reciprocal transformation \cite{liu1989singular,book:green2012linear}. While this interpretation has been used for the theoretical investigations of \SPA, its relevance for the algorithmic implementation was not further investigated so far.
In this section, we propose a definition of the reciprocal transformation in a slightly modified setting, considering \LTI systems with general matrices $\bE$ as opposed to the special case $\bE = \bI$ used in literature. We show that our extension inherits the fundamental properties that have been derived  in \cite{liu1989singular,book:green2012linear,guiver2018generalised} for the special case.

\begin{definition}\label{def:ReciprocalDeSys} 
For a \LTI system as in \eqref{DesSys}, the reciprocal system, i.e., the reciprocal transformed counterpart of the \LTI system, is defined as
\begin{align*}
\re{E} \, \dot{\re{x}}(t) &= \re{A} \, \re{x}(t) + \re{B} \, \re{u}(t). \\
\re{y}(t) &= \re{C} \, \re{x}(t) + \re{D} \,\re{u}(t), \\[-1cm]
\end{align*}
\begin{align*}
 \hspace{0.3em} \text{with }\hspace{0.3em}& \re{A} = \bE\bA^{-1} \bE,  &\re{B} = \bE \bA^{-1} \bB ,\hspace{0.8cm} &     \re{E} = \bE,\hspace{2cm}  \\
 & \re{C} = -\bC \bA^{-1} \bE,  &\re{D} =  \bD - \bC \bA^{-1}  \bB.  &
\end{align*}
The reciprocal transfer function is defined by  $\TFre(s) = \re{C} (s \re{E} - \re{A})^{-1} \re{B} + \re{D}$ for $s \in \IC$.
\end{definition}

From a theoretical point of view, one could directly invert the matrix $\bE$ to end up in the standard setting. However, this step needs to be avoided for the efficient implementation of \SPA in the general large-scale setting, cf.~\Cref{sec:LowRankSPA} below.

\begin{proposition}\label{lem:recInverse} 
The reciprocal transformed system of an asymptotically-stable \LTI system as in \eqref{DesSys} is well-defined and asymptotically stable. Moreover, by applying twice the reciprocal transformation to a \LTI system, the \LTI system is obtained.
\end{proposition}
The latter result follows by straightforward calculus. Next, the relation in the frequency domain between the two transfer functions is provided; this will clarify the origin of the term "reciprocal".
\begin{theorem} 
The transfer function $\TF$ of the \LTI system in \eqref{DesSys} and the reciprocal transfer function $\TFre$ are reciprocal to each other, in the sense that for any $s \in \IC^+ $, it holds
\begin{align*}
	\TFre(s) = \TF\left( \frac{1}{s} \right) \hspace{1.0cm} \text{and, respectively,} \hspace{1cm} \TFre \left( \frac{1}{s} \right) = \TF(s).
\end{align*}
\end{theorem}
\begin{proof}
We first prove that $(1/s \bE - \bA)^{-1} = (-\re{E}^{-1} \re{A}) [ \re{E}^{-1} +  (s \re{E}-  \re{A})^{-1} \re{A} \re{E}^{-1}]$ as a preliminary step. The latter holds due to
\begin{align*}
(1/s \bE - \bA)  (-\re{E}^{-1} \re{A}) \bigl[ \re{E}^{-1} +  (s \re{E} -  \re{A})^{-1} \re{A} \re{E}^{-1} \bigr] \hspace{4.3cm} \\
= 1/s (s \re{E} - \re{A}) \bigl[ \re{E}^{-1} +  (s \re{E} -  \re{A})^{-1} \re{A} \re{E}^{-1} \bigr] 
 =  1/s  [  (s \re{E} - \re{A})  \re{E}^{-1} + \re{A} \re{E}^{-1} ]  =\bI.
\end{align*}
With this equality, it follows that
\begin{align*}
\TF({1}/{s} ) &=  \bD+ \bC ( {1}/{s} \bE - \bA)^{-1}\bB  \\
 &= \bD+ \bC (-\re{E}^{-1} \re{A}) \bigl[ \re{E}^{-1} +  (s \re{E} -  \re{A})^{-1} \re{A} \re{E}^{-1} \bigr] \bB \\
 &= [\bD - \bC \re{E}^{-1} \re{A}  \re{E}^{-1} \bB] - \bC [\re{E}^{-1} \re{A}  (s \re{E} -  \re{A})^{-1} \re{A} \re{E}^{-1} ] \bB \\
 &= \re{D} + [-\bC \re{E}^{-1} \re{A} ] (s \re{E} -  \re{A})^{-1} [\re{A} \re{E}^{-1}  \bB]  =  \re{D}  + \re{C} \bigl(s \re{E} - \re{A} \bigr)^{-1} \re{B} = \TFre(s).
\end{align*}
The other relation follows by interchanging the role of the \LTI and the reciprocal system, cf. \Cref{lem:recInverse}.
\end{proof}
The main results of this paper are based on the following alternative characterization of \SPA that is also  graphically illustrated in \Cref{fig:comDiagramSPA}.

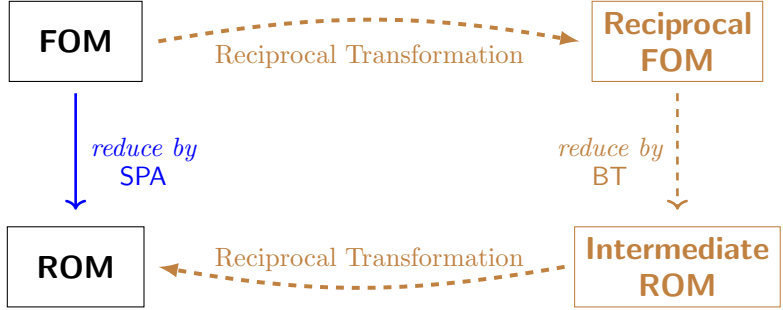
\begin{figure}[tb]
	\caption{Commutative diagram and alternative characterization of \SPA (following the brown dashed route) that is stated in \Cref{lem:SPAbyBT}. \label{fig:comDiagramSPA}}

\begin{tikzpicture}[auto]
\node [rect]                     (SFOM) at (1.5, 3)  {\text{\FOM}};
\node [rect]                     (SROM) at (1.5, 0)  {\text{\ROM}};

\node [rect2,align=center,brown]                     (SreFOM) at (9.5, 3)  {\text{Reciprocal} \\ \text{\FOM}};
\node [rect2,align=center,brown]                     (SreROM) at (9.5, 0)  {\text{Intermediate} \\ \text{\ROM}};

\node [align=center]                     (tBT) at (2.4, 1.4)  {\textit{\color{blue}reduce by}\\[-0.1em] {\color{blue}\SPA}};
\node [align=center]                     (tBT) at (8.6, 1.4)  {\textit{\color{brown}reduce by}\\[-0.1em] {\color{brown}\BT}};
\draw [blue, line width=0.35mm, ->] (SFOM.south)+(0,-0.15) -- +(0,-1.7cm) node [pos=0.5, above, black] {};
\draw [dashed,brown, line width=0.35mm, ->] (SreFOM.south)+(0,-0.15) -- +(0,-1.7cm) node [pos=0.5, above, black] {};
\Edge[Direct,bend=10, color = brown, style = dashed](2.6, 3)(8.2, 3);
\Edge[Direct,bend=10, color = brown, style = dashed](8.0, 0)(2.6, 0);
\node [align=center,brown]                     (tBT) at (5.4, 2.8)  {Reciprocal Transformation};
\node [align=center,brown]                     (tBT) at (5.4, 0.1)  {Reciprocal Transformation};

\end{tikzpicture}

\end{figure}

\begin{proposition}\label{lem:SPAbyBT}
For an asymptotically-stable \LTI system with Hankel singular values fulfilling $\sigma_r > \sigma_{r+1}$, the \ROM of dimension $r$ obtained by \SPA, defined in \cref{eq:spa-matrices}, can be equivalently obtained by a procedure that incorporates the next steps:
\begin{enumerate}[i$_)$]
	\item Apply the reciprocal transformation to the \FOM.
	\item Construct an intermediate reduced model of dimension $r$ for the reciprocal system, using \BT.
	\item Apply the reciprocal transformation to the intermediate reduced model.
\end{enumerate}
\end{proposition}
The proof of the proposition can be found in \cite{liu1989singular, guiver2018generalised} for the special case $\bE = \bI$. Since the {\ROM} obtained by a balancing-related method does not depend on the state realization of the \LTI system, it is clear that the proposition also holds for our consistently extended definition of a reciprocal system.

The last result presented here provides the relation between the Gramians of an \LTI and those of the reciprocal system.
\begin{proposition} \label{lem:recGramians} 
The controllability and observability Gramians of the \LTI system coincide with those of its reciprocal system in \Cref{def:ReciprocalDeSys}.
\end{proposition}
\begin{proof}
It can be proven that the Lyapunov equations of the \LTI system and those of the reciprocal system are equivalent. Since these equations have unique solutions, this implies the equality of the Gramians. For example, the Lyapunov equation characterizing the controllability Gramian $\re{P}$ of the reciprocal system reads, by definition,
$
\re{A} \re{P} \re{E}^T + \re{E} \re{P} \re{A}^T + \re{B}\re{B}^T = \bv{0}.
$
This is equivalent to
\begin{align*}
(\bE \bA^{-1} \bE) \re{P} \bE ^T + \bE  \re{P} (\bE \bA^{-1} \bE)^T +  (\bE \bA^{-1} ) \bB\bB^T (\bE \bA^{-1} )^T = \bv{0}.
\end{align*}
By multiplying the latter equation with $(\bE \bA^{-1})^{-1}$ from the left and $(\bE \bA^{-1})^{-T}$ from the right, the first Lyapunov equation in \eqref{lyapFacDes} is obtained, which implies $\re{P}= \bP$. The equality of the  observability Gramians for the \LTI system and the reciprocal system can be shown to hold true in a similar manner.
\end{proof}

%% file: secLowRank.tex
\section{Low-rank implementations of balancing-related methods}
\label{sec:LowRnkImp}

In this section, we describe the typical strategies for implementing a low-rank variant of \BT, based on the balancing-free square-root method. Afterwards, our first main result is provided, consisting of the new balancing-free implementation of \SPA. Its main advantage is that it allows for a low-rank variant similar to its \BT counterpart. The gained computational efficiency is illustrated at a large-scale numerical example.

\subsection{Low-rank implementation of Balanced Truncation} \label{subsec:BT-lowRank}

Due to the asymptotic stability of the original \LTI system (i.e., the \FOM), the Gramians are symmetric, positive semi-definite matrices. Thus, one may compute a factorization
\begin{equation} \label{eq:lyapFact}
\bP = \bU \bU^T \quad \mbox{and} \quad \ \bQ = \bL \bL^T
\end{equation}
with  $\bL, \bU \in \mathbb{R}^{n \times n}$, e.g., via a Cholesky factorization.
This provides the essential elements for an implementation of \BT by the square-root method outlined in~\Cref{alg:origbt}. It is to be noted that, instead of first computing the Gramians explicitly, and then the factors $\bL$ and $\bU$, one can instead compute the latter quantities directly, by means of the Hammarling method implemented in the Matlab function \textsf{lyapchol} or by a sign-function based solver, cf.,~\cite{morBenQQ00}. This circumvents potential numerical issues that may arise when computing Cholesky factors of already computed Gramians.
Another option for optimizing the computational steps is to replace the true Cholesky factors in \eqref{eq:lyapFact} by approximated low-rank factors, i.e., $\bP \approx \breve{\bU} \breve{\bU}^T$ and  $\bQ \approx \breve{\bL} \breve{\bL}^T$, with $\breve{\bU} \in \mathbb{R}^{n \times r_U}, \breve{\bL} \in \mathbb{R}^{n \times r_L}$, which have significantly less columns than the square-root factors, i.e,. $r_L, r_U \ll n$. Then, the computation a full SVD for a large-scale matrix is averted, since 	$\bL^T \bE \bU \in \mathbb{R}^{n \times n}$ is replaced by $\breve{\bL}^T \bE \breve{\bU} \in \mathbb{R}^{r_L \times r_U}$. Such low-rank  factors of the Gramians can be computed by means of, among others, Lyapunov equation iterative solvers based on the matrix sign function, on low-rank Alternating Direction Implicit (ADI) methods, Smith-type methods, etc.; a more detailed analysis on such approaches can be found in \cite{penzl1999cyclic,li2002low,benner2013efficient,ValSim16}.

The low-rank adaption of the square-root implementation of \BT is remarkably straight forward: the exact factors of the Gramians have to be replaced by the approximated low-rank factors in ~\Cref{alg:origbt}, see, e.g., \cite{morBenQQ00,morGugL05,baur2008gramian}. This very simple adaptation follows essentially from the projection-based nature of \BT. The underlying projection spaces consisting of the most observable and controllable states are efficiently approximated by the low-rank factors.
The practical relevance of the low-rank implementations is not to be underestimated, as large-scale sparse systems appear in various applications from computational fluid dynamics and mechanical, electronic, chemical or civil engineering. Specific applications case are numerous, ranging from microelectronics, aerodynamics, acoustics, to electromagnetics, neuroscience, or chemical process optimization, to enumerate only a few; see, e.g., ~\cite{book:dimred2003,BOCW17,morBenGQetal21b}.

\begin{algorithm}[tbhp] 
\footnotesize 
	\caption{Balanced truncation (\BT) (square-root/low-rank implementation)}  
	\label{alg:origbt}                                        
	\algorithmicrequire~\LTI system described by matrices $\bE , \bA \in \mathbb{R}^{n \times n}, \ \bB\in \mathbb{R}^{n \times m}$, $\bC\in \mathbb{R}^{p \times n}$, $\bD\in \mathbb{R}^{p \times m}$.
	
	\algorithmicensure~\ROM:   $\bA_r \in \mathbb{R}^{r \times r},  \bB_r\in \mathbb{R}^{r \times m}, \bC_r \in \mathbb{R}^{p \times r}$, and \,$\bD_r \in \mathbb{R}^{p \times m}$.
	
	\begin{algorithmic} [1]                                        
		\STATE \label{computeUL} Compute the  Lyapunov factors $\bU, \bL \in \mathbb{R}^{n\times n}$ from \cref{eq:lyapFact} and pick a 
		reduced dimension $1\leq r\leq \min(\mathsf{rank}(\bU),\mathsf{rank}(\bL))$.
		\STATE  \label{svdstep} Compute the SVD of the matrix $ \bL^T \bE \bU$, partitioned as follows 
		\begin{equation*} 
		\bL^T \bE \bU = \left[ \begin{matrix}
		\bZ_1 & \bZ_{2}
		\end{matrix}  \right] \left[ \begin{matrix}
		\bS_1 & \\[1ex] & \bS_{2}
		\end{matrix}  \right] \left[ \begin{matrix}
		\bY_1^T \\[1ex] \bY_{2}^T
		\end{matrix}  \right],
		\end{equation*}
		where $\bS_1 \in \mathbb{R}^{r \times r} \ \mbox{and} \ \bS_{2} \in \mathbb{R}^{(n-r) \times (n-r)}$.
		\STATE \label{WrVrstep}  
		Construct the model reduction bases $\bW_r = \bL \bZ_1 	\bS_1^{-1/2}$ and $\bV_r = \bU \bY_1 	\bS_1^{-1/2}$
		\STATE \label{project}  The reduced-order system matrices are given by
		\begin{align*}
		\bA_r = \bW_r^T \bA \bV_r , \hspace{0.2cm} \bB_r = \bW_r^T \bB , \hspace{0.2cm} \bC_r = \bC \bV_r , \hspace{0.2cm} \text{and} \hspace{0.2cm} \bD_r =\bD.
		\end{align*}
	\end{algorithmic}
\end{algorithm}

\subsection{Low-rank implementation of Singular Perturbation Approximation} \label{sec:LowRankSPA}

In contrast to \BT, the reduction by \SPA has no direct interpretation in the projection framework. Moreover, from the characterization \eqref{eq:spa-matrices}, which is explicitly used in the algorithms proposed in  \cite{benner2000singular, baur2008gramian,Varga91}, it was not clear how to explicitly derive a low-rank implementation of \SPA.

We therefore consider the alternative procedure suggested by \Cref{lem:SPAbyBT} for the derivation of our new algorithm. In the upcoming, a basic non-efficient variant is sketched. Then, the novel efficient implementation, which allows for a low-rank adaption, is derived from the latter by a few crucial modifications.

The basic procedure suggested by \Cref{lem:SPAbyBT} reads as follows. First the reciprocal transformation is applied to the \FOM to obtain the reciprocal  system $(\re{A},\re{B}, \re{C}, \re{D}, \re{E})$. Then, the reduction bases $\re{W}_r, \re{V}_r \in \IR^{n \times r}$ for the reciprocal system are determined (using \BT on the reciprocal system), and the intermediate \ROM is  obtained by a projection step, i.e.,
\begin{align} \label{eq:spaProj}
	\re{A}_r  = \re{W}_r^T \re{A} \re{V}_r, \qquad  \re{B}_r  = \re{W}_r^T \re{B}, \qquad \re{C}_r  = \re{C} \re{V}_r, \hspace{0.2cm} \text{and} \hspace{0.2cm} \re{D}_r = \re{D}.
\end{align}
As a final step, the reciprocal transformation is applied to the intermediate reduced model $(\re{A}_r,\re{B}_r, \re{C}_r, \re{D}_r)$; cf.\,the last dashed brown error in \Cref{fig:comDiagramSPA}.

The computational issue with the latter procedure is that it requires the reciprocal transformation of the \FOM, implying an inversion of the \FOM matrix $\bA$. Further note that the reciprocal system has dense state matrices even if the \FOM is sparse. Thus, the procedure has to be modified in such a way that the explicit determination of the full-order reciprocal system is omitted.

A first crucial observation in this direction is that the construction of $\re{W}_r$ and $\re{V}_r$ does not depend on the reciprocal system itself but only on its Gramians. By \Cref{lem:recGramians}, the reciprocal system has the same Gramians as the \FOM and, thus, the reduction bases for applying \BT to the reciprocal system are the same as for applying \BT to the \FOM, i.e., $\re{W}_r = \bW_r$ and $\re{V}_r=\bV_r$  (with $\bW_r$ and $\bV_r$ as in \Cref{alg:origbt}).
The second observation is that the projection step \eqref{eq:spaProj} for the intermediate \ROM does not require the full reciprocal system either but can be obtained from solving a system of linear equations of moderate size. For example, the construction of the state matrix  $\re{A}_r$ of the intermediate \ROM can be done according to
\begin{align} 
\begin{aligned} \label{eq:spaPstep}
 \re{A}_r  &= \re{W}_r^T \re{A} \re{V}_r = \bW_r^T  \re{A}   \bV_r  =  \bW_r^T \bE \bA^{-1} \bE  \bV_r \\
& \quad  \Longleftrightarrow \re{A}_r =\bW_r^T \bE \re{A}_V , \qquad \text{with $\re{A}_V$ solving }   \bA \re{A}_V = \bE \bV_r.
\end{aligned}
\end{align}
The efficient construction of the other matrices $\re{B}_r$, $\re{C}_r$ and $\re{D}_r$ follows similarly.

The new and efficient implementation of \SPA, which follows by taking into account all the previously mentioned steps, is summarized in \Cref{alg:spaSR}. Similarly as for \BT, one can replace the exact factors of the Gramians ($\bL$ and $\bU$) in step 1 of the algorithm with low-rank approximate factors to obtain an efficient implementation for the large-scale setting.

\begin{algorithm}[htp] 
\footnotesize 
	\caption{Singular Perturbation Approximation (\SPA) (square-root/low-rank)}  
	\label{alg:spaSR}                                        
	\algorithmicrequire~\FOM realization: $\bE, \bA \in \mathbb{R}^{n \times n}, \ \bB\in \mathbb{R}^{n \times m}$, $\bC\in \mathbb{R}^{p \times n}$, $\bD\in \mathbb{R}^{p \times m}$.
	
	\algorithmicensure~\ROM:   $\bA_r \in \mathbb{R}^{r \times r},  \bB_r\in \mathbb{R}^{r \times m}, \bC_r \in \mathbb{R}^{p \times r}$, and \,$\bD_r \in \mathbb{R}^{p \times m}$.
	
	\begin{algorithmic} [1]                                        
		\STATE  Compute the Lyapunov factors $\bU, \bL \in \mathbb{R}^{n\times n}$ for approximating the solutions of \eqref{lyapFacDes} and pick a 
		reduced dimension $1\leq r\leq \min(\mathsf{rank}(\bU),\mathsf{rank}(\bL))$.
		\STATE  Compute the SVD of the matrix $ \bL^T \bE \bU$, partitioned as follows 
		\begin{align*}
		\bL^T \bE \bU = \left[ \begin{matrix}
		\bZ_1 & \bZ_{2}
		\end{matrix}  \right] \left[ \begin{matrix}
		\bS_1 & \\[1ex] & \bS_{2}
		\end{matrix}  \right] \left[ \begin{matrix}
		\bY_1^T \\[1ex] \bY_{2}^T
		\end{matrix}  \right],
		\end{align*}
		where $\bS_1 \in \mathbb{R}^{r \times r} \ \mbox{and} \ \bS_{2} \in \mathbb{R}^{(n-r) \times (n-r)}$.
		
		\STATE 
		Construct the model reduction bases $\bW_r = \bL \bZ_1 	\bS_1^{-1/2}$ and $\bV_r = \bU \bY_1 	\bS_1^{-1/2}$.

		\STATE \label{alg:spaSR-Pstep} 
		Calculate $\re{A}_V =  \bA^{-1} (\bE \bV_r)$ and $\re{B}_A = \bA^{-1} \bB$. 
		Construct an intermediate \ROM approximating the reciprocal system:
\begin{align*}
\re{A}_r =  \bW_r^T \bE \re{A}_V,  \hspace{0.2cm} 
 \re{B}_r = \bW_r^T \bE  \re{B}_A,  
\re{C}_r = -\bC \re{A}_V,  \hspace{0.2cm}  \text{and} \hspace{0.2cm} 
\re{D}_r =\bD - \bC \re{B}_A .
\end{align*}		
\vspace{-1em}		
		\STATE \label{alg:spaSR-Trafstep}
		 Get the \ROM by using the reciprocal transformation onto the intermediate \ROM, i.e.,
\begin{align*}
\		{\bA}_r = \re{A}_r^{-1} ,  \hspace{0.2cm}
		{\bB}_r =   \re{A}_r^{-1} \re{B}_r  ,  \hspace{0.2cm}
		{\bC}_r = - \re{C}_r \re{A}_r^{-1} ,  \hspace{0.2cm}  \text{and} \hspace{0.2cm}
{\bD}_r = \re{D}_r - \re{C}_r \re{A}_r^{-1} \re{B}_r .
\end{align*}
	\end{algorithmic}
\end{algorithm}

\begin{remark}
Let us give two practical hints for ~\Cref{alg:spaSR}. Of course, the inverse of $\bA$ is not formed in step \ref{alg:spaSR-Pstep}, but a few linear equations are solved instead. This is done by following the steps indicated in \eqref{eq:spaPstep}.
Moreover, the last step in the algorithm is most efficiently implemented using a successive evaluation of the matrices using the following ordering. First ${\bA}_r = \re{A}_r^{-1}$ is evaluated, then ${\bB}_r =   {\bA}_r \re{B}_r$, followed by ${\bC}_r = - \re{C}_r {\bA}_r$. Finally, ${\bD}_r = \re{D}_r + {\bC}_r  \re{B}_r$ is determined.
\end{remark}

\subsection{Numerical study for low-rank implementation} \label{subsec:numLowRank}

It has been shown by several authors that the use of low-rank Lyapunov equation solvers saves orders of magnitude in a large-scale setting. This allows to solve with problem sizes that are computationally infeasible for direct (dense) solvers, based mainly on algorithms proposed by Bartels-Stewart and Hammarling, and recent adaptations (see \cite{sorensen2003direct} for an overview). These results directly transfer to \BT and, as we showcase in this section, also to our newly proposed \Cref{alg:spaSR}, which is the first low-rank implementation of \SPA (as far as the authors are aware of). We refer to the method as '\texttt{Alg.\,2,\,low-rank}' in the following and compare it to its dense counterpart '\texttt{Alg.\,2,\,dense}' as well as to the \texttt{Matlab} built-in method \texttt{balred} realizing \SPA.

\begin{figure}[h tb]
\text{$ $} \hspace{0.6cm} \underline{{\SPA and \BT} with $r = 12$} \hspace{2.7cm} \underline{{\SPA and \BT} with $r = 36$}\\
\includegraphics[scale = 0.38]{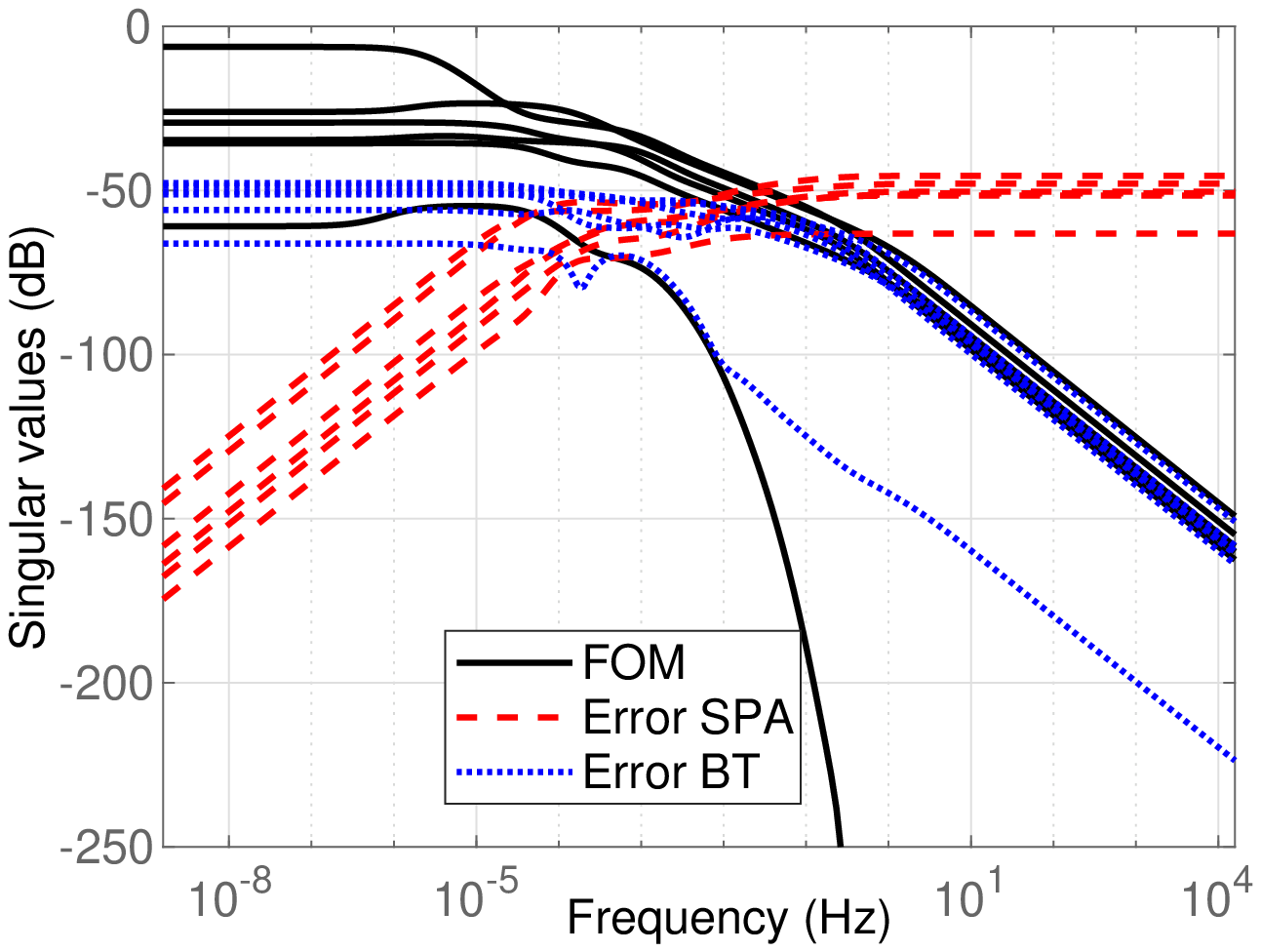} \hspace{1.5cm}
\includegraphics[scale = 0.38]{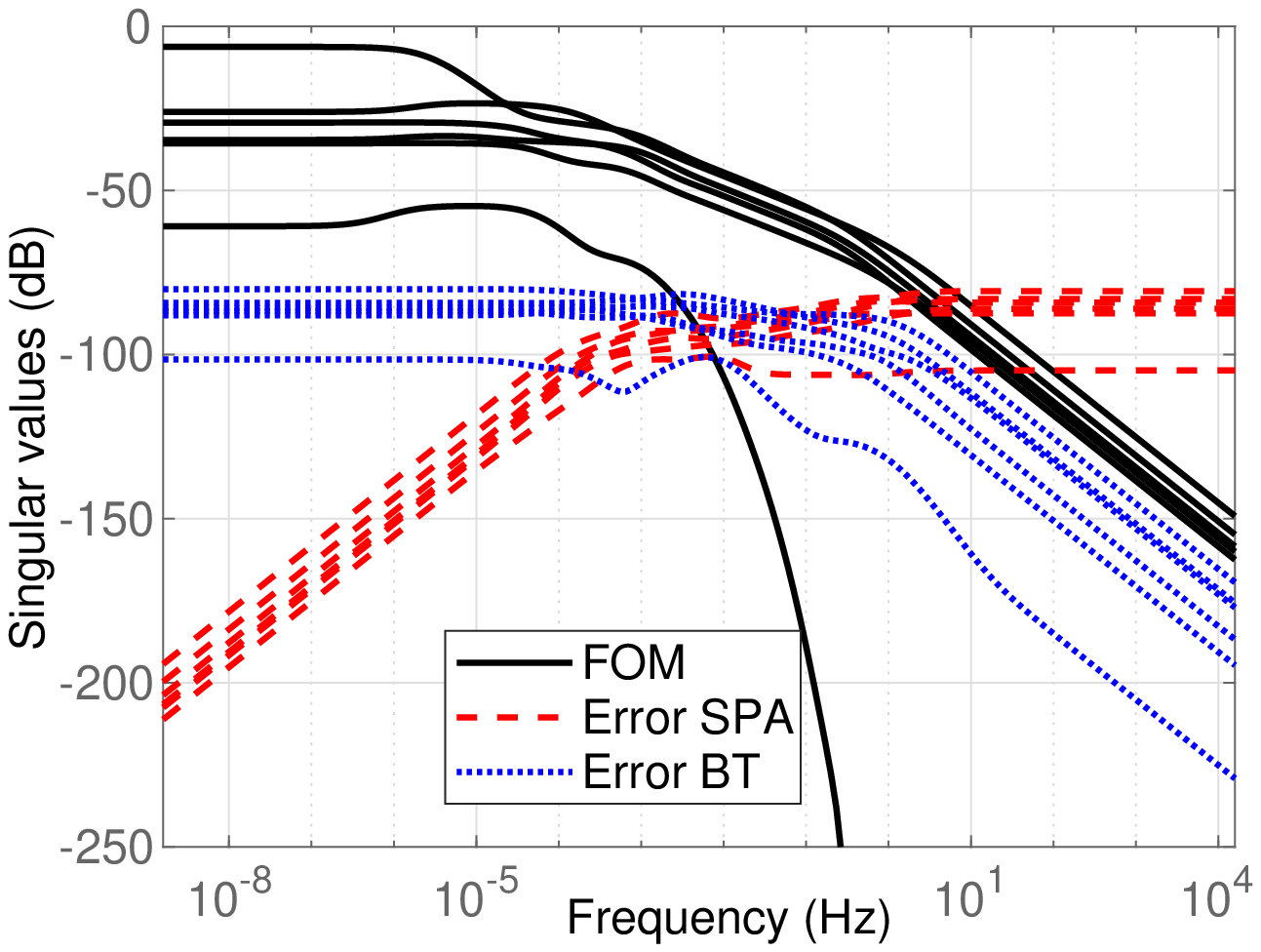}\\[0cm]
\vspace{-4mm}
\caption{\rail (with $n = 20\,209$). Frequency response of \FOM and the errors related \SPA and \BT; using low-rank operations.\label{fig:rail-freq}}                     \vspace{-3mm}                           
\end{figure}

\begin{SCfigure}
\raggedleft
  \caption{\rail with varying dimension $n$. CPU times for \SPA realized by Algorithm 2 (either using low-rank or dense operations) and \texttt{Matlab} built-in function \texttt{balred} ($r=12$ for all {\ROM}s).\label{fig:rail-times}}
\includegraphics[scale = 0.42]{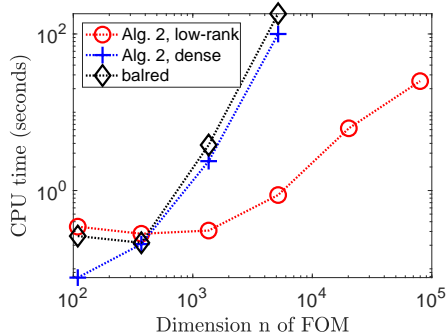}%
\vspace{-3mm}
\end{SCfigure}

The numerical studies are performed, from here on, for a benchmark problem described in \cite[Section 19]{book:dimred2003} and referred to as \rail, which is given by a semi-discretized heat transfer model of a steel rail. Depending on the underlying mesh resolution of the discretization, the \FOM has a dimension $n \in \{109$, $371$, $1\,357$, $5\,177$, $      20\,209$, $79\,841\}$. Both the benchmark data as well as the low-rank (ADI-based) Lyapunov equation solvers originate from the \texttt{MESS} toolbox \cite{SaaKB-mmess}. The numerical results have been generated using \texttt{MATLAB} Version 9.12 (R2022a) on an computer running with an Intel Core i7-8700 CPU with 32.0GB. To allow a better reproducibility of the reported results, the code and the benchmark data are provided in \cite{code:LiGoSPA}.

Notably, regardless of which of the three implementations was chosen for the realization of \SPA, almost no effect on the quality of the resulting \ROM was observed for the test cases with $n\leq 5\,177$. More precisely, the deviations with respect to the ${\cH_{\infty}}$ norm were below $10^{-10}$. However, larger dimensions could not be considered for the dense solvers, since they did not run to completion on the running machine.

The qualitative behavior of the \FOM and the approximation errors for \SPA and \BT are illustrated in the frequency domain in \Cref{fig:rail-freq} for two different choices of the \ROM order $r \in \{ 12,36\}$. As to be expected, \SPA performs better in the lower frequency range, while \BT is better in the higher frequency range. The maximum frequency-domain errors of the methods are fairly similar, ranging at about $-50 db$ for $r=12$ and  at $-75 db$ for $r =36$.

However, the computational times of the different \SPA implementations vary significantly as illustrated in \Cref{fig:rail-times}. While the two variants using dense solvers are almost the same in this respect, the low-rank implementation \texttt{Alg.\,2,\,low-rank} is already seven times faster for a \FOM dimension of $n = 1\,357$, i.e., $0.3$ seconds compared to $2.8$ seconds. The difference becomes even more pronounced for larger dimensions. The sparse solver is even faster for $n = 79\,841 $ than \texttt{balred} for $n= 5\,177$. 
(\texttt{Alg.\,2,\,low-rank} for $n = 79\,841 $: $25.0$~seconds versus \texttt{balred} for $n= 5\,177$: $181.2$~seconds). Larger dimensions are not feasible for \texttt{balred} due to the internal memory limits of the computing machine. It is to be noted that the very large memory demand of the dense solvers is a strong limitation in the large-scale setting.



%% file: secQuad.tex
\section{Data-driven implementation of balancing-related methods} \label{sec:QuadBT}

In \cite{gosea2022data}, it was shown that \BT can be realized (approximately) in a non-intrusive, data-driven manner. There, ''data'' are transfer function evaluations. This is fundamentally different to the other data-driven \BT approaches in \cite{moore1981principal,willcox2002balanced, Ro05,opmeer2011model}, which require snapshots of the full state. The main idea underlying the approach in \cite{gosea2022data} is to make use only implicitly of quadrature approximations of the two system Gramians for the construction of a low-order approximately balanced model.

The novel contribution of this part is to derive a similar data-driven adaptation of \SPA (or \QSPA). Since it follows similar principles as \cite{gosea2022data}, we set the stage in \Cref{subsec:ddBT} by going over the basic construction underlying the latter. In \Cref{sec:DataDrivenSPA} \QSPA is derived, and its numerical performance is examined in \Cref{sec:Numerics2}.

Since we are interested in realization-free approaches, we may assume $\bE = \bI$ from here on. This modification carries over to the \FOM, i.e., instead of \eqref{DesSys}, we have
\begin{align*}
\dot{\bx}(t) &= \bA\, \bx(t) + \bB\, \bu(t), \hspace{1.2cm}
\by(t) = \bC \bx(t) + \bD \bu(t).
\end{align*}

\subsection{Data-driven Balanced Truncation} \label{subsec:ddBT}

Following the \BT implementation  in \Cref{alg:origbt} (with $\bE = \bI$), square-root factors $\bU$, and $\bL$ of the Gramians are defined, and their product is approximated by a singular value decomposition, according to
\begin{align*}
	\bL^T\bU \approx \bZ_1 \bS_1 \bY_1^T, \hspace{1cm} \bS_1 \in \IR^{r,r}.
\end{align*}
By the definition of the reduction bases $\bW_r = \bL \bZ_1 	\bS_1^{-1/2}$ and $\bV_r = \bU \bY_1 	\bS_1^{-1/2}$ (Step~3 in the algorithm), the matrices of the \ROM realization are expressed as
\begin{align*} 
			\bA_r &= \bW_r^T \bA \bV_r=	\bS_1^{-1/2}\bZ_1^T\,(\bL^T\bA\bU)\,\bY_1\bS_1^{-1/2},  
			&		\bB_r = \bW_r^T \bB=\bS_1^{-1/2}\bZ_1^T\,(\bL^T\bB), \\[2mm]
			 \bC_r &= \bC \bV_r=(\bC\bU)\,\bY_1\bS_1^{-1/2}.
			 	\end{align*}
By the latter relations, we observe that the \ROM is fully characterized in terms of the following key quantities:
\begin{align}\label{eq:keyquant}
\bL^T \bU, \hspace{0.3cm} \bL^T\bA\bU,  \hspace{0.3cm} \bL^T\bB \hspace{0.3cm} \text{and} \hspace{0.3cm} \bC\bU.
\end{align}
These matrices can be approximated well from input-output data, as it is shown next. 

The starting point is a quadrature approximation of the Gramians in the frequency domain. Let $\imunit$ be the complex unit with $\imunit^2 = -1$. Then, the controllability Gramian $\bP$ and the observability Gramian $\bQ$. (i.e., the solutions to \eqref{lyapFacDes}) can be expressed as
\begin{align*} 
\begin{aligned} 
\bP &=
\frac{1}{2\pi}\int_{-\infty}^{\infty} (\imunit \zeta\bI -\bA)^{-1} \bB \bB^T (-\imunit \zeta\bI^T -\bA^T)^{-1} d \zeta, \\
\bQ &= \frac{1}{2\pi}\int_{-\infty}^{\infty} (-\imunit \omega\bI -\bA^T)^{-1} \bC \bC^T (\imunit \omega\bI -\bA)^{-1} d \omega.
\end{aligned}
\end{align*}
We consider first a numerical quadrature rule that approximates the frequency integral defining the controllability Gramian, producing an approximate Gramian $\quadP$:
\begin{align} \label{quad_P}
\bP \approx \quadP = \sum_{k=1}^{\np} \rho_k^2  (\imunit \omega_k \bI -\bA)^{-1} \bB \bB^T (-\imunit \omega_k \bI^T -\bA^T)^{-1}
\end{align}
Hereby, $\omega_k$ and $\rho_k^2$ denote the quadrature nodes and quadrature weights, respectively. Thus, $\rho_k$ are, by construction, the square-roots of the quadrature weights.
Evidently, we may decompose the quadrature-based Gramian approximation  as  $\quadP = \quadU \quadU^*$ with a square-root factor $\quadU \in \mathbb{C}^{n \times \np m}$ defined as 
\begin{equation} \label{quad_U}
\quadU = \left[ \,
\rho_1  (\imunit \omega_1 \bI -\bA)^{-1} \bB \quad 
\cdots \quad \rho_\np  (\imunit \omega_\np \bI-\bA)^{-1} \bB  \right].
\end{equation}
Note that both $\bP$ and its quadrature approximation, $\quadP$, are real-valued matrices, yet the explicit square-root factor, $\quadU$, is overtly complex and subsequent computation engages complex floating point arithmetic. 

Similarly, a quadrature approximation of $\bQ$ is defined by
\begin{align*}
\bQ \approx \quadQ = \sum_{j=1}^{\np} \qwe_j^2  (-\imunit \qpo_j \bI^T -\bA^T)^{-1} \bC^T \bC (\imunit \qpo_j \bI -\bA)^{-1},
\end{align*}
with quadrature points $\qpo_j$ and quadrature weights $\qwe_j^2$. (For convenience we assume the number of quadrature points to be the same for $\bQ$ and $\bP$.)
The corresponding square-root factor decomposition reads
$\quadQ = \quadL \quadL^*$ where
\begin{equation}  \label{quad_L}
\quadL^* = \left[ \begin{matrix}
\qwe_1 \bC (\imunit \qpo_1\bI -\bA)^{-1} \\
\vdots \\ \qwe_\np \bC (\imunit \qpo_\np \bI -\bA)^{-1}
\end{matrix}  \right] \in \mathbb{C}^{\np p \times n}.
\end{equation}

For a concise notation in the multi-input multi-output setting, the following two technical definitions are useful.

\begin{definition}\label{def:blockMatkj}
	Let $\bX \in \IC^{\np p \times \np m }$. Then, for $1 \leq k,j \leq \np$, we say that the $(k,j)$th block $(p,m)$ entry of matrix $\bX$ is a matrix in $\IC^{p \times m}$, denoted with $\bX_{k,j}$, for which its $p$ rows are a subset of the rows of matrix $\bX$, indexed from $ (k-1)p+1$ to $k p$. Additionally, the $m$ columns of the matrix $\bX_{k,j}$ are a subset of the columns of the matrix $\bX$, indexed from $ (k-1)m+1$ to $k m$.
\end{definition}
It is to be noted that in \Cref{def:blockMatkj}, if $m = p =1$ (corresponding to the SISO case), then $\bX_{k,j}$ is nothing else but the $(k,j)$th scalar (or $(1,1)$) entry of matrix $\bX \in \IC^{\np \times \np}$.
\begin{definition}\label{def:blockVeckj}
	Let $\bX \in \IC^{\np p \times  m }$. Then, for $1 \leq k \leq \np$, we say that the $k$th block $(p,m)$ entry of matrix $\bX$ is a matrix in $\IC^{p \times m}$, denoted with $\bX_{k}$, for which its $p$ rows are a subset of the rows of matrix $\bX$, indexed from $ (k-1)p+1$ to $k p$. Additionally, the $m$ columns of matrix $\bx_{k}$ are exactly the same as those of matrix $\bX$. (An equivalent definition can be formulated for $\bX \in \IC^{p \times  \np m }$; however, for brevity sake, we skip the details here.)
\end{definition}

\begin{algorithm}[htp] 
\footnotesize 
	\caption{Quadrature-based Balanced Truncation (\QBT)}  
	\label{alg:quadbt}                                     
	\algorithmicrequire~-{Transfer} function of \FOM: $\TF(s)$;\\
	\hspace*{5mm} -Two sets of quadrature rules for approximating an integral of the form $1/(2\pi) \int_{-\infty}^{\infty} f(s) ds$:\\
	\hspace*{7mm} given by $\omega_k$, $\rho_k$, respectively $\qpo_j$, $\qwe_j$ for for $k,j=1,2,\ldots,\np$ (assuming $\omega_k\neq \qpo_j$)\\
	\hspace*{5mm} reduced dimension $r$, $1\leq r\leq \np \min(m,p)$.
	
	\algorithmicensure~{\ROM}:  $\qdr{\bA}_r\in\mathbb{R}^{r \times r}, \ \qdr{\bB}_r \in \mathbb{R}^{r \times m}, \qdr{\bC}_r \in \mathbb{R}^{p \times r}$ and \,$\qdr{\bD}_r \in \IR^{p \times m }.$
	\begin{algorithmic} [1]   
		\STATE Determine feedtrough term $\bD = \lim_{s\to \infty} \TF(s)$.
		\STATE \label{sample} Evaluate the strictly proper transfer function $\spTF:s \mapsto \TF(s) - \bD$ at $\{\spTF(\imunit \omega_j)\}_{j=1}^{\np}$ and $\{\spTF(\imunit \qpo_k)\}_{k=1}^{\np}$.  Using the samples and the quadrature weights, 
		construct $\boldsymbol{\quadLL},\boldsymbol{\quadMM}$, 
		$\quadLb$, and $\quadcU$  as in~\Cref{prop:quadMatrices}.
		\STATE Compute the SVD of matrix $\quadLL \in \IC^{\np p \times \np m}$:
		\begin{align*}
		\quadLL  = \left[ \begin{matrix}
		\quadZ_1 & \quadZ_{2}
		\end{matrix}  \right] \left[ \begin{matrix}
		\quadS_1 & \\ & \quadS_{2}
		\end{matrix}  \right] \left[ \begin{matrix}
		\quadY_1^* \\ \quadY_{2}^*
		\end{matrix}  \right],
		\end{align*}
		where $\quadS_1 \in \mathbb{R}^{r \times r}$ and $\quadS_{2} \in \mathbb{R}^{(\np p-r) \times (\np m-r)}$.
		\STATE Construct the reduced order matrices: 
\begin{align*}
\qdr{\bA}_r = \quadS_1^{-1/2} \quadZ_1^*  \quadMM \ \quadY_1   \quadS_1^{-1/2},  \hspace{0.2cm} 
\qdr{\bB}_r =  \quadS_1^{-1/2} \quadZ_1^* \quadLb,
\qdr{\bC}_r =  \quadcU^T \quadY_1   \quadS_1^{-1/2},  \hspace{0.2cm}  \text{and} \hspace{0.2cm} 
\qdr{\bD}_r = \bD
\end{align*}			
	\end{algorithmic}
\end{algorithm}

By the latter considerations and noticing that $\TF(\infty)= \bD$, it can be seen that the key quantities in \eqref{eq:keyquant} can be approximated by transfer function evaluations. 
\begin{proposition} \label{prop:quadMatrices}
Let $\quadU$ and $\quadL$ be as defined in \eqref{quad_U} and \eqref{quad_L}, whereby the quadrature points $\omega_k \neq \qpo_j$ are distinct for $1\leq k,j \leq \np$.  Also, let
\begin{align*}
 	\spTF(s) = \bC (s  \bI - \bA)^{-1} \bB = \TF(s) - \lim_{\tilde{s} \to \infty} \TF(\tilde{s}).
\end{align*}
	Define the matrices $\quadLL =\quadL^*  \quadU \in \IC^{\np p \times \np m}$  and 
	 $\quadMM =\quadL^* \bA \quadU \in \IC^{\np p \times \np m}$.\\
	 Then, for $1\leq k,j\leq\np$, the $(k,j)$th block $(p,m)$ entries of matrices $\quadLL$ and $\quadMM$, respectively read as (following \Cref{def:blockMatkj}):
\begin{align*} 
	\quadLL_{k,j} &= - \rho_k \qwe_j
	\displaystyle \frac{\spTF(\imunit\omega_k) - \spTF(\imunit\omega_j)}{\imunit\omega_k - \imunit \omega_j},
	\\
	\quadMM_{k,j} &= 
	- \rho_k \qwe_j\displaystyle \frac{\imunit\omega_k \spTF(\imunit\omega_k) - \imunit \qpo_j\spTF(\imunit\omega_j)}{\imunit\omega_k - \imunit \omega_j}.
	\end{align*}
	Likewise, defining $\quadLb = \quadL^*\bB
	\in \IC^{\np p \times m}$ and $\quadcU^T = \bC^T \quadU\in \IC^{p \times \np m}$, and by following \Cref{def:blockVeckj}, we find 
	\begin{align*} 
\quadLb_k & = \rho_k \spTF(\imunit \omega_k),  \hspace{0.3cm} \text{and} \hspace{0.3cm} 
\quadcU_j =    \qwe_j\spTF(\imunit \omega_j), \hspace{0.3cm} \text{for} \ 1 \leq k,j \leq \np.
	\end{align*}
\end{proposition}
More details and various extensions (using time-domain data, for discrete-time systems, etc.) can be found in \cite{gosea2022data}. Based on \Cref{prop:quadMatrices}, the data-driven adaptation of \BT can be implemented, as it is shown in \Cref{alg:quadbt}. A proof of the proposition can be found  in \Cref{app:proofQBT}; this is indeed useful for setting up the stage for the  derivations of the novel \QSPA presented in the next section. 




\subsection{Data-driven Singular Perturbation Approximation} \label{sec:DataDrivenSPA}

The formulation of the novel data-driven \QSPA implementation is presented here. Similarly as for the low-rank implementation of \SPA (\Cref{sec:LowRankSPA}), we rely on a construction based on the alternative characterization of \Cref{lem:SPAbyBT} and illustrated in \Cref{fig:comDiagramSPA}.
First, an intermediate \ROM $(\rer{A},\rer{B}, \rer{C}, \rer{D})$ approximating the reciprocal counterpart of the \FOM is constructed. Then, by applying the reciprocal transformation of this low-dimensional intermediate \ROM, we obtain the \QSPA reduced system approximating the \FOM. 

However, the intermediate \ROM is constructed from data here, following a similar quadrature approximation as in the data-driven \BT, cf.~\Cref{subsec:ddBT}.
This approximation requires a careful choice of quadrature rule.  To derive an appropriate one, we rewrite the frequency representations of the reciprocal Gramians $\re{P}$ and $\re{Q}$ in terms of evaluations of the original transfer function in \eqref{eq:TF_def}.
Let
\begin{align*}
 \Kre: \IC \to \IC^{n,n}, \qquad \Kre(s) = (s \bI - \re{A})^{-1},
\end{align*}
so that the strictly proper part of the reciprocal transfer function $s \mapsto \TFre(s) - \TFre(\infty)$ can be equivalently expressed as $s \mapsto\re{C} \Kre(s) \re{B}$.
Similarly to the derivation of the data-driven \BT, we consider the frequency representations of the Gramians. 
The reciprocal controllability Gramian reads
\begin{align*}
\begin{split}
	\re{P} &=
	\frac{1}{2\pi}\int_{-\infty}^{\infty} \Kre(\imunit \zeta) \re{B} \re{B}^T \Kre(-\imunit \zeta)^{T}  d \zeta= \frac{1}{2\pi}\int_{-\infty}^{\infty} \Kre \left( \frac{1}{\imunit \zeta} \right) \re{B} \re{B}^T \Kre
	\left( \frac{1}{-\imunit \zeta} \right)^{T} \frac{1}{\zeta^2}  d \zeta,
\end{split}
\end{align*}
whereby the last equality follows by using integration with the substitution  $s \to -1/s$ on 
$
\FF(\zeta):= \frac{1}{2\pi} \Kre(\imunit \zeta) \re{B} \re{B}^T \Kre(-\imunit \zeta)^{T}
$,
according to
\begin{align*}
	\int_{-\infty}^{\infty} \FF(\zeta) d\zeta &= \int_{-\infty}^{0} \FF(\zeta) d\zeta  + \int_{0}^{\infty} \FF(\zeta) d\zeta \\
	& =   \int_{0}^{\infty} \FF\left( -\frac{1}{\zeta} \right) \frac{1}{\zeta^2}  d\zeta  + \int_{-\infty}^{0} \FF\left( -\frac{1}{\zeta} \right) \frac{1}{\zeta^2}  d \zeta =
\int_{-\infty}^{\infty} \FF\left( -\frac{1}{\zeta} \right) \frac{1}{\zeta^2}  d\zeta.
\end{align*}
Based on the reformulation, we consider the following approximation by quadrature,
\begin{align*}
\begin{split}
	\re{P} &\approx \qdr{\re{P}} = \sum_{k=1}^{\np} \left(\frac{\rho_k}{\omega_k}\right)^2  \Kre \left( \frac{1}{\imunit \omega_k} \right) \re{B} \re{B}^T \Kre
	\left( \frac{1}{-\imunit \omega_k} \right)^{T},
\end{split}
\end{align*}
with quadrature nodes $\omega_k$ and weights $\rho_k^2$. This is in analogy to the approximation provided by \eqref{quad_P}. We define a square-root factor $\requadU$, which fulfills $\requadP = \requadU \requadU^*$, as
\begin{equation} \label{eq:SpaQuad_U}
\requadU = \left[ \,
\frac{\rho_1}{\omega_1} \Kre \left( \frac{1}{\imunit \omega_1} \right) \re{B} \quad 
\cdots \quad \frac{\rho_\np}{\omega_\np} \Kre\left( \frac{1}{\imunit \omega_\np} \right)\re{B}  \right] \in \mathbb{C}^{n \times \np m}.
\end{equation}
Similarly, the reciprocal observability Gramian is approximated by quadrature
\begin{align*} 
\begin{split}
\re{Q} \approx \requadQ = \sum_{k=1}^{\np} \left(\frac{\qwe_j}{\qpo_j}\right)^2  \Kre \left( \frac{1}{-\imunit \qpo_j} \right)^T \re{C}^T \re{C} \Kre
	\left( \frac{1}{\imunit \qpo_j} \right),
\end{split}
\end{align*}
with quadrature points and weights $\qpo_j$  and $\qwe_j^2$, respectively. The corresponding square-root factor decomposition reads
$\requadQ = \requadL \requadL^*$, where
\begin{equation}  \label{eq:SpaQuad_L}
\requadL^* = \left[ \begin{matrix}
\frac{\qwe_1}{\qpo_1} \re{C} \Kre \left( \frac{1}{\imunit \qpo_1} \right)\\
\vdots \\ 
\frac{\qwe_\np}{\qpo_\np} \re{C} \Kre \left( \frac{1}{\imunit \qpo_\np} \right) \\
\end{matrix}  \right] \in \mathbb{C}^{\np p \times n}.
\end{equation}

Certain key quantities of the reciprocal system can be directly approximated from data, as the following result shows. The construction is similar to the one used in \Cref{prop:quadMatrices}.

\begin{proposition} \label{prop:quadSPA}
	Let $\requadU$ and $\requadL$ be as defined in 
	\eqref{eq:SpaQuad_U} and \eqref{eq:SpaQuad_L}, whereby the quadrature-points $\omega_k \neq \qpo_j$ are distinct and nonzero for $1\leq k,j \leq \np$.
	Also, let
	\begin{align*}
	\TFzer: \IC \to \IC^{p,m}, \qquad  \TFzer({s}) =  \TF(s) - \TF(0).
	\end{align*}
	Define the matrices $\requadLL =\requadL^*  \requadU \in \IC^{\np p \times \np m}$  and 
	 $\requadMM = \requadL^* \re{A} \requadU \in \IC^{\np p \times \np m}$. Then, for $1\leq k,j\leq\np$, the $(k,j)$th block $(p,m)$ entries of matrices $\requadLL$ and $\requadMM$, respectively read (following \Cref{def:blockMatkj}) as
\begin{align*} 
	\requadLL_{k,j} &= 
 - \rho_k \qwe_j 
	\displaystyle \frac{\TFzer(\imunit\omega_k) - \TFzer(\imunit\qpo_j)}{\imunit\omega_k -  \imunit \qpo_j},
	\\
	\requadMM_{k,j} &= 
	- \rho_k \qwe_j \displaystyle \frac{(\imunit\omega_k)^{-1} \TFzer(\imunit\omega_k) - (\imunit\qpo_j)^{-1} \TFzer(\imunit\qpo_j)}{\imunit\omega_k -  \imunit \qpo_j}.
	\end{align*}
	Likewise, defining $\requadLb = \requadL^*\re{B}
	\in \IC^{\np p \times m}$ and $\requadcU^T = \re{C} \requadU\in \IC^{p \times \np m}$
 we find
	\begin{align*} 
	\requadLb_k & =  \frac{\rho_k}{\omega_k} \TFzer(\imunit \omega_k), \hspace{0.3cm} \text{and} \hspace{0.3cm}
	\requadcU_j =    \frac{\qwe_j}{\qpo_j} \TFzer(\imunit \qpo_j), \hspace{0.3cm} \text{for} \ 1 \leq k,j \leq \np.
	\end{align*}
\end{proposition}

\begin{proof}
\def\VarA{{\nu}}
\def\VarB{{\mu}}
Only for this proof, we introduce the abbreviations $\VarA = (\imunit\omega_k)^{-1}$ and $\VarB =  (\imunit\qpo_j)^{-1}$. 
A small calculation shows that
\begin{align*}
	\Kre(\VarA) \Kre(\VarB) = \frac{1}{\VarA - \VarB} \Kre(\VarA) \left[ (\VarA \bI - \re{A}) -  (\VarB \bI - \re{A})  \right] \Kre(\VarB) =  - \frac{1}{\VarA - \VarB} \left( \Kre(\VarA)-\Kre(\VarB) \right).
\end{align*}
By the latter, and $ \re{C} \Kre(\VarA ) \re{B}  = \TFzer({1/\VarA}) =  \TFzer(\imunit\omega_k) $, and $ \re{C} \Kre(\VarB ) \re{B}  =\TFzer(\imunit\qpo_j) $, it follows
\begin{align*}
\requadLL_{k,j} &=  (\rho_k/\omega_k)  (\qwe_j/\qpo_j) \, \re{C} \Kre(\VarA) \Kre(\VarB) \re{B} =   - \rho_k \qwe_j \frac{1 }{\omega_k \qpo_j(\VarA - \VarB)}  \re{C} \left[ \Kre(\VarA)-\Kre(\VarB) \right] \re{B} \\
&= - \rho_k \qwe_j \frac{\TFzer(\imunit\omega_k) - \TFzer(\imunit\qpo_j)}{\imunit\omega_k- \imunit \qpo_j},
\end{align*}
whereby $\omega_k \qpo_j (\VarA-\VarB) = \imunit\omega_k-\imunit\qpo_j$ was used in the last step. Moreover using the relation
\begin{align*}
	\Kre(\VarA) \re{A} \Kre(\VarB) = -\frac{1}{\VarA - \VarB}  \left[\VarA  \Kre(\VarA)- \VarB \Kre(\VarB) \right],
\end{align*}
we can show in a similar manner that
\begin{align*} 
\requadMM_{k,j}  &= (\rho_k/\omega_k) (\qwe_j/\qpo_j) \, \re{C} \left(\Kre(\VarA) \re{A} \Kre(\VarB)  \right) \re{B} = -\rho_k \qwe_j \frac{1 }{\omega_k \qpo_j(\VarA - \VarB)}  \re{C} \left(\VarA  \Kre(\VarA)- \VarB \Kre(\VarB) \right) \re{B} \nonumber \\
&= - \rho_k \qwe_j \displaystyle \frac{(\imunit\omega_k)^{-1} \TFzer(\imunit\omega_k) - (\imunit\qpo_j)^{-1} \TFzer(\imunit\qpo_j)}{\imunit\omega_k -  \imunit \qpo_j}.
\end{align*}
The explicit representations of $\requadLb $ and $\requadcU$ (as data matrices) are easily derived, since
\begin{align*}
	\requadLb_k &= (\rho_k/\omega_k)   \, \re{C} \Kre(\imunit \omega_k) \re{B} = \frac{\rho_k}{\omega_k} \TFzer(\imunit \omega_k),\ \ \
	\requadcU_j =  (\qwe_j/\qpo_j) \, \re{C}  \Kre(\imunit \qpo_j) \re{B} = \frac{\qwe_j}{\qpo_j} \TFzer(\imunit \qpo_j).
\end{align*}
\end{proof}

The main contribution of this section, i.e., the data-driven, realization-free implementation of \SPA is summarized in \Cref{alg:quadspq_new}.

\begin{remark}
A generalized version of \Cref{prop:quadSPA} omitting the assumption $\omega_k \neq \qpo_j$ for all $j,k$ can be derived, but it requires evaluations of the derivative of the transfer function, see~\Cref{app:QuadBalancing} for details.
\end{remark}
{
\begin{algorithm}[htp] 
\footnotesize 
	\caption{Quadrature-based Singular Perturbation Approximation (\QSPA)}  
	\label{alg:quadspq_new}                                     
	\algorithmicrequire~-{Transfer} function of \FOM: $\TF(s)$;\\
	\hspace*{5mm} -Two sets of quadrature rules for approximating an integral of the form $1/(2\pi) \int_{-\infty}^{\infty} f(s) ds$:\\
	\hspace*{7mm} given by $\omega_k$, $\rho_k$, respectively $\qpo_j$, $\qwe_j$ for for $k,j=1,2,\ldots,\np$ (assuming $\omega_k\neq \qpo_j$)\\	
	\hspace*{5mm} -reduced dimension $r$, $1\leq r\leq \min(\np p, \np m)$.
	
	\algorithmicensure~{\ROM}:  $\qdr{\bA}_r\in\mathbb{R}^{r \times r}, \, \qdr{\bB}_r \in \IR^{r \times m}, \, \qdr{\bC}_r \in \IR^{p \times r}$, and \,$\qdr{\bD}_r \in \IR^{p \times m}$
	\begin{algorithmic} [1]   
		\STATE Determine the moment $\bD_0 =  \TF(0)$.		
\STATE \label{sample2} Evaluate the modified transfer function $\TFzer:s \mapsto \TF(s) - \bD_0$ at $\{\TFzer(\imunit \omega_k)\}_{k=1}^{\np}$ and $\{\TFzer(\imunit \qpo_j)\}_{j=1}^{\np}$. Using the samples and the quadrature weights, construct the data matrices $\requadLL,\requadMM$, 
$\requadLb$, and $\requadcU$  as in~\Cref{prop:quadSPA}.
\STATE Compute the SVD of matrix $\requadLL \in \IC^{\np p \times \np m}$:
\begin{equation*} 
\requadLL  = \left[ \begin{matrix}
\quadZ_1 & \quadZ_{2}
\end{matrix}  \right] \left[ \begin{matrix}
\quadS_1 & \\ & \quadS_{2}
\end{matrix}  \right] \left[ \begin{matrix}
\quadY_1^* \\ \quadY_{2}^*
\end{matrix}  \right],
\end{equation*}
where $\quadS_1 \in \mathbb{R}^{r \times r}$ and $\quadS_{2} \in \mathbb{R}^{(\np p-r) \times (\np m-r)}$.
\STATE Construct the intermediate reduced order matrices: 
\begin{align*}
\rer{A} = \quadS_1^{-1/2} \quadZ_1^*  \requadMM \ \quadY_1   \quadS_1^{-1/2}
\rer{B} =  \quadS_1^{-1/2} \quadZ_1^* \requadLb ,  \hspace{0.2cm}
\rer{C} =  \requadcU^T \quadY_1   \quadS_1^{-1/2} ,  \hspace{0.2cm}  \text{and} \hspace{0.2cm}
\rer{D} = {\bD_0} .
\end{align*}	
\vspace{-1em}
\STATE		 Get the \ROM by using the reciprocal transformation onto the intermediate \ROM, i.e.,
\begin{align*}
\		\qdr{\bA}_r = \rer{A}^{-1} ,  \hspace{0.2cm}
		\qdr{\bA}_r = \rer{A}^{-1} ,  \hspace{0.2cm}
		\qdr{\bC}_r = - \rer{C} \rer{A}^{-1}  ,  \hspace{0.2cm}  \text{and} \hspace{0.2cm}
\qdr{\bD}_r = \rer{D} - \rer{C} \rer{A}^{-1} \rer{B} .
\end{align*}		
	\end{algorithmic}
\end{algorithm}
}

\subsection{Connections of \QSPA to the Loewner framework} \label{subsec:loewner}

In this section, we aim at revealing some connections that the newly-developed method, \QSPA, bears with the Loewner framework in \cite{ajm07}. Note that a connection of \QBT and the Loewner framework has already been made in \cite[Section 3.4]{gosea2022data}.

First, by comparing \Cref{prop:quadSPA} and \Cref{prop:quadMatrices}, it can be shown that the data matrix $\requadLL$ used in \QSPA and the $\quadLL$ used in \QBT are exactly the same. They have the interpretation of a diagonally scaled Loewner matrix.
The other data matrices, i.e., $\quadMM$ in \QBT  and $\requadMM$ in \QSPA relate to (scaled) shifted Loewner matrices. In \QBT, the shifts are given by $\imunit\omega_k$, $1\leq k\leq \np$, while in \QSPA, there one can find the 'inverted shifts' $(\imunit\omega_k)^{-1}$ instead.

This new context for \QSPA implies a few observations. First, \QSPA has a similar computational complexity as the Loewner framework. Choosing more data points $\np$ will, in general, imply increasing closeness to the intrusively computed \ROM, but will also increase the dimension of the data matrices. Second, the quality of the \ROM strongly depends on the quality of the data. When using data sets with perturbed  (noisy) measurements, or that contain redundant data (not rich enough to capture the essential properties), the \QSPA faces similar challenges as the Loewner framework. We intend to study these aspects in more detail, for future works.

As was pointed out in preceding publications, the numerical conditioning of Loewner matrices can be a relevant issue. While the application of quadrature weights (as done in \QBT and \QSPA) may improve the conditioning, regularization strategies might still be necessary. Of course, the SVD is an essential step that acts as implicit regularizer, removing the redundancies, provided the truncation order is small enough. Other possible strategies are described in \cite{binder2022texttt}, while the robustness of the Loewner approach to noise or perturbations, was investigated in \cite{zhang2021factorization,Drmac2022}. 

Also, we note that \Cref{alg:quadspq_new} requires a complex SVD, which implies that the resulting \ROM realization is, in general complex-valued. However, under the additional assumption that the quadrature points used in the approximation of either Gramian is chosen in complex-conjugate pairs, a real-valued \ROM can be enforced. This can be done in a similar manner as was proposed in \cite[Section A.1]{AntBG20} for the Loewner framework, and in \cite[Section 4.1]{gosea2022data} for \QBT.

%% file: secNumQuad.tex
\section{Numerical study for the data-driven implementation} \label{sec:Numerics2}

In this section, we compare the performance of the newly-proposed quadrature-based, data-driven implementation of \SPA according to \Cref{alg:quadspq_new}, i.e., \QSPA, with the following model reduction methods from the literature: the intrusive  \texttt{MATLAB} built-in implementations of Singular Perturbation Analysis (\SPA) and Balanced Truncation (\BT), and  the quadrature-based implementation of Balanced Truncation, i.e., \QBT from \cite{gosea2022data}, realized by \Cref{alg:quadbt}.

We consider two benchmark examples, one of low dimension and one of higher dimension; both are available in the MOR-Wiki\footnote{\url{https://morwiki.mpi-magdeburg.mpg.de/morwiki/index.php/Category:Benchmark}} database, cf.~\cite{book:dimred2003}:
\begin{enumerate}
	\item 	
	The \build system models a motion problem in the Los Angeles University hospital building, with 8 floors each
	having 3 degrees of freedom, namely displacements in the $x$ and $y$ directions, and rotation. The model has $n=48$ states, $m=1$ input, and $p=1$ output. We refer to \cite{morChaV02} for more details.

\item 
	The \iss system models the structural response of the Russian Service 12A Module of the International Space Station (ISS). The model has $n=1412$ states, $m=3$ inputs, and $p=3$ outputs. We refer to \cite{morGugAB01} for more details. 
\end{enumerate}

For all experiments, the quadrature weights required for \QSPA and \QBT are determined according to the given transfer function measurements and the trapezoid quadrature rule. Other choices of quadrature rules are, of course, possible; we refer, e.g., to \cite{benner2010balanced,BERTRAM202166,gosea2022data}, which consider other more involved quadrature schemes.
Extended details are given below but can also be directly comprehended from the \texttt{MATLAB} code used to generate the numerical results, provided in \cite{code:LiGoSPA}.
\subsection{\build model} \label{sec:Numerics21}

In the first experiment, we choose $\np =100$ logarithmically-spaced distributed (sampling) points in the range $\left[10^{0}, 10^2\right] \cdot \imath$. These will act as quadrature nodes. We compare the results of applying the two data-based methods against their intrusive counterparts, i.e., \SPA and \BT. We fix $r=18$ as the dimension of all four {\ROM}s, for the first experiment. In the left pane of~\Cref{fig:fig1}, we depict the magnitude of the frequency response for the original system and approximation errors for the 2 {\ROM}s computed with \SPA and \QSPA. In the right pane of~\Cref{fig:fig1}, we depict the magnitude of the frequency response for the original system and approximation errors for the two {\ROM}s computed with \BT and \QBT. We notice that the \SPA-based methods provide worse approximation quality in the higher frequency range, while the \BT-based methods provide worse approximation quality in the lower frequency range. Additionally, it is to be noticed that the error curve corresponding to $\QSPA$ faithfully reproduces that of $\SPA$. The same can be said about $\BT$.

\begin{figure}[h tb]
\includegraphics[scale= 0.4]{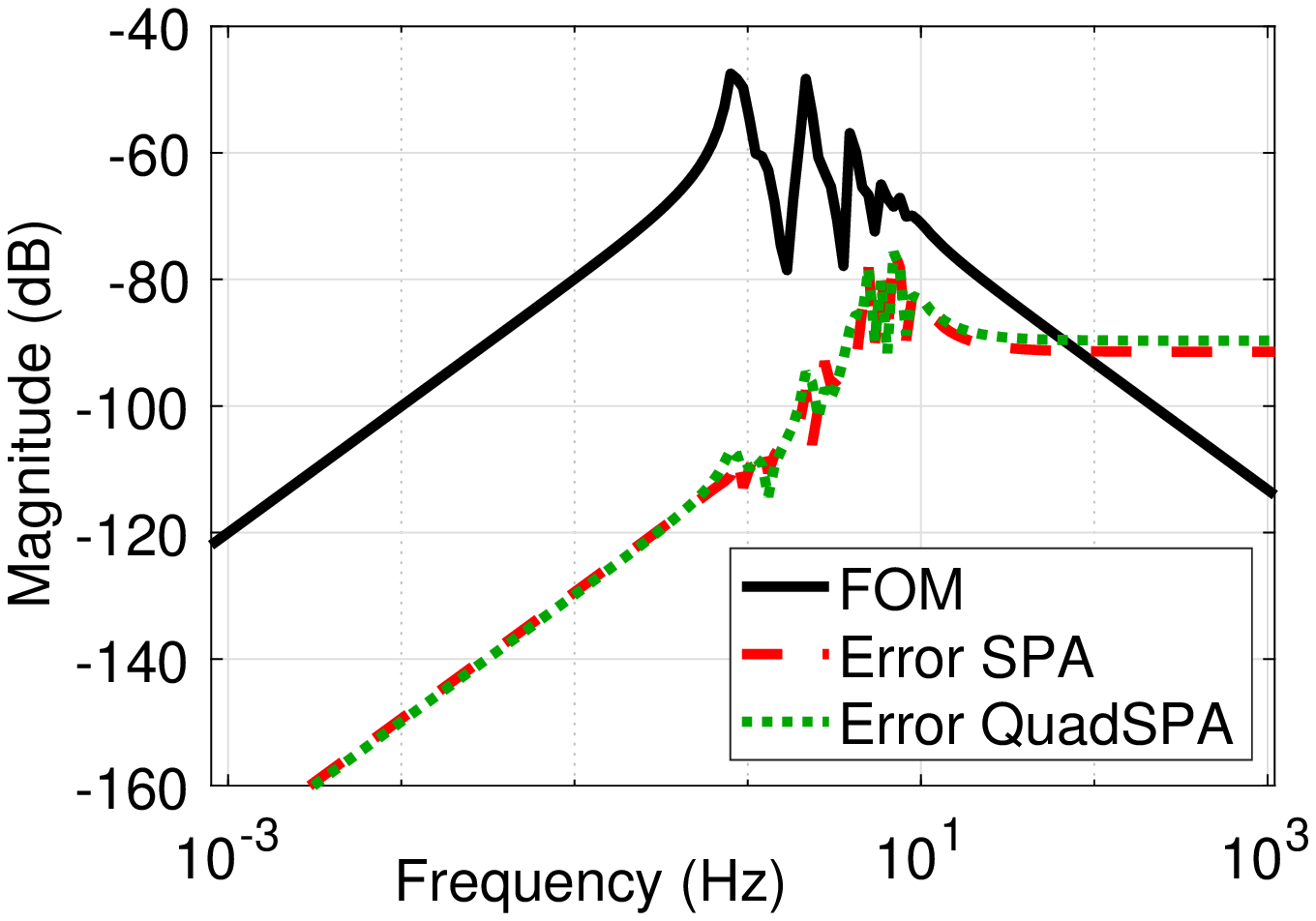} \hspace{1.5cm}
\includegraphics[scale= 0.4]{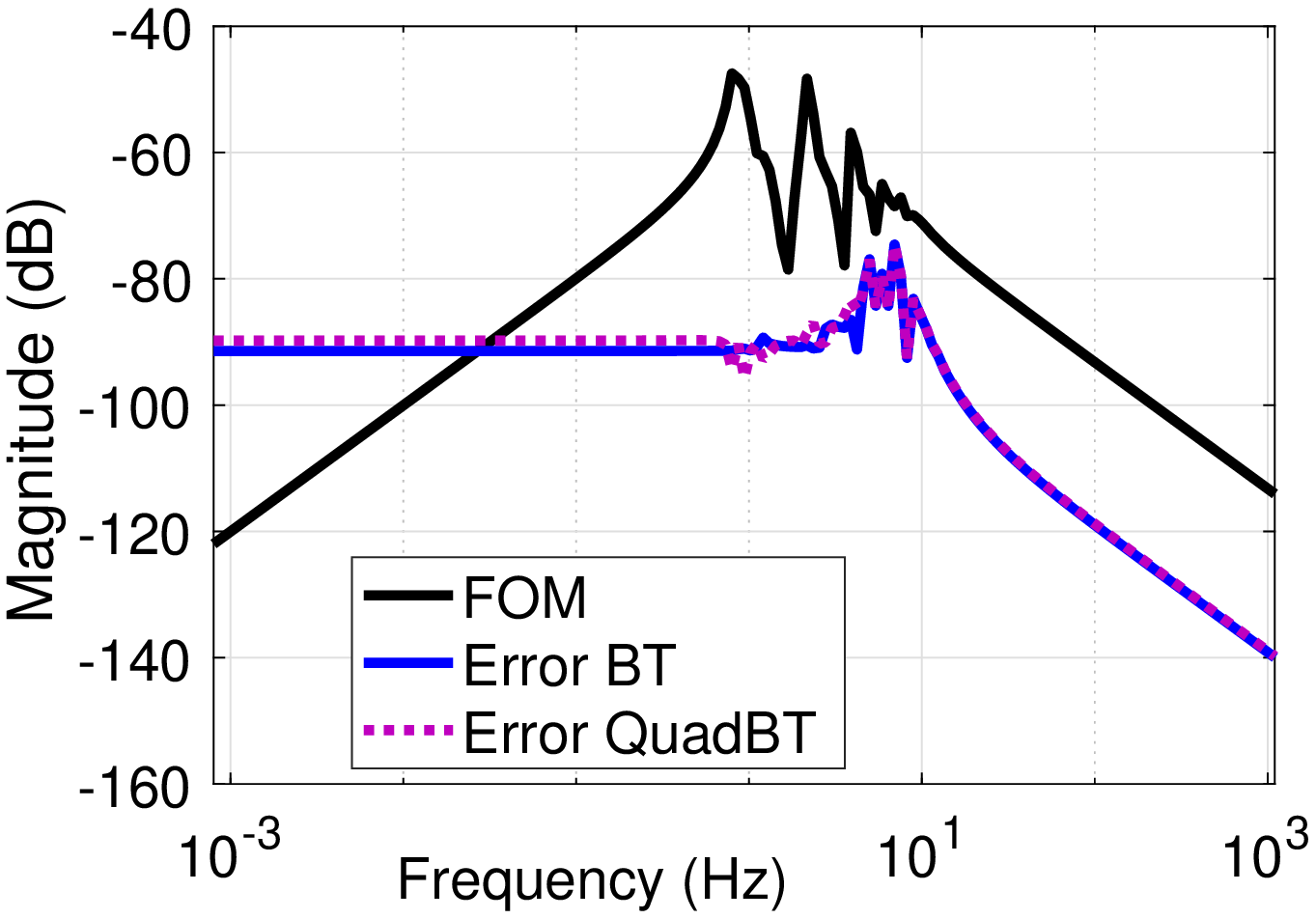} \hspace{1.5cm}
	\caption{\build. Frequency response of \FOM and the errors using \SPA and \QSPA (left) and using \BT and \QBT (right); $\np = 100$ and $r=18$ is chosen for all {\ROM}s.}
	\label{fig:fig1}
\end{figure}

In the next experiment, we vary the dimension $r$ of the \ROM for all four methods in increments of $6$. The number of data points is fixed to the same value as before, i.e., $\np = 100$. We compute the $\cH_{\infty}$ norm of the error systems produced by the four model reduction methods, scaled by the $\cH_{\infty}$ norm of the original system. The numerical results are shown in~\Cref{tab:table1}. As expected, the approximation errors decrease as $r$ increases. Additionally, the quadrature-based methods produce errors comparable to the intrusive counterparts. Actually, for $r \in \{18,24\}$, the $\QSPA$ method outperforms $\SPA$ in term of the $\cH_{\infty}$ norm approximation. However, for $r=12$, the error provided by $\QSPA$ is much higher than that of $\SPA$ (due to a pair of rogue poles), and a similar behavior is noticed for $r=30$, both for $\QSPA$ and for $\QBT$.

\begin{table}[h]
	\centering
	\caption{\build. Relative approximation errors in the $\cH_{\infty}$ norm for varying \ROM order $r$ (\QSPA and \QBT, with $\np=100$).}
	\label{tab:table1}
	\small
	\begin{tabular}{ |p{1.5cm}||p{1.7cm}|p{1.7cm}|p{1.7cm}|p{1.7cm}| p{1.7cm}| }
		\hline
		& $r = 6$ & $r = 12$ & $r = 18$ & $r = 24$  & $r = 30$ \\
		\hline
		\SPA   & $2.4185 \cdot 10^{-1}$ & $ 9.3060 \cdot 10^{-2}$ & $3.7846 \cdot 10^{-2}$ & $1.0922 \cdot 10^{-2}$ & $9.0847 \cdot 10^{-4}$\\
		\hline
		\QSPA &  $2.4971 \cdot 10^{-1}$ &  $5.7480 \cdot 10^{-1}$ &  $3.6713 \cdot 10^{-2}$ & $1.0836 \cdot 10^{-2}$ & $ 3.9822 \cdot 10^{-3}$\\
		\hline
		\BT  &  $ 2.3084 \cdot 10^{-1}$ &  $1.0317 \cdot 10^{-1}$ &  $3.8312 \cdot 10^{-2}$ & $ 1.0613 \cdot 10^{-2}$ & $ 9.4410 \cdot 10^{-4}$\\
	    \hline
		\QBT &  $2.7935 \cdot 10^{-1}$ &  $1.0442 \cdot 10^{-1}$ &  $3.8193 \cdot 10^{-2}$ & $1.0285 \cdot 10^{-2}$ & $  4.5524 \cdot 10^{-3}$\\
		\hline
	\end{tabular}
	
\end{table}

What we generally observe, is that a higher reduction order $r$ requires more quadrature nodes $\np$ in \QSPA and \QBT in order to have a similar fidelity as the intrusive methods \SPA and \BT, respectively.
For example, in the latter experiment, we require about $\np = 300$ points (three times more than before) to ensure that the quadrature-based methods reproduce the quality of the intrusive methods for $r=30$. With this choice of parameters, the relative $\mathcal{H}_\infty$ error for \QSPA is $9.3103 \cdot 10^{-4}$ for \QSPA and $ 7.9862 \cdot 10^{-4}$ for \QBT, which is very close to the errors of \SPA and, interestingly, even slightly better than for \BT, respectively, cf.\,\Cref{tab:table1}.


For the next experiment, we vary instead the $\np$ parameter between $[20,100]$. We fix the dimension of the {\ROM}s as $r = 18$. We compute again the relative approximation errors in the $\cH_\infty$ norm, for the two quadrature-based methods. The numerical results are shown in~\Cref{tab:table2}. As expected, for increasing values of $\np$, the relative approximation errors are decreasing for both $\QSPA$ and $\QBT$ and typically, they reach the quality of their intrusive counterparts (\SPA and \BT). Notably, the quadrature-based methods can lead to better results in exceptional cases. This is observed here for $\np = 100$. However, this does not guarantee, by no means, that for $\np>100$, the same trend is valid; e.g., when choosing $\np = 120$, we noticed that \BT outperforms \QBT.

\begin{table}[h!]
	\centering
	\caption{\build. Relative approximation errors in the $\cH_{\infty}$ norm for varying quadrature points $\np$, compared to relative errors of intrusive {\ROM}s ($r = 18$ for all {\ROM}s).}
	\label{tab:table2}
	\small
	\begin{tabular}{ |p{1.5cm}||p{1.7cm}|p{1.7cm}|p{1.7cm}|p{1.7cm}| p{1.7cm}| }
		\hline
		& $\np = 20$ & $\np = 30$ & $\np = 50$ & $\np = 70$  & $\np = 100$ \\
		\hline
		\QSPA   & $3.4571 \cdot 10^{-1}$ & $3.2915 \cdot 10^{-1}$ & $6.0762 \cdot 10^{-1}$ & $7.0414 \cdot 10^{-2}$ & $3.6713 \cdot 10^{-2}$\\
		\hline
		\QBT &  $7.9048 \cdot 10^{-1}$ &  $3.8459 \cdot 10^{-1}$ &  $1.9527 \cdot 10^{-1}$ & $9.9991 \cdot 10^{-2}$ & $ 3.8193 \cdot 10^{-2}$\\
		\hline
	\end{tabular}
		\small
\begin{tabular}{ |p{3.2cm}||p{3.2cm}|}
	\hline
 \SPA: $3.7846 \cdot 10^{-2}$ &   \BT: $ 3.8312 \cdot 10^{-2}$  \\
	\hline
\end{tabular}	
\end{table}

\vspace{-2mm}	


\subsection{\iss model} \label{sec:Numerics22}

For the \iss model, we first choose $\np = 1000$ nodes, chosen to be logarithmically-spaced in $[10^{-1},10^2] \cdot \imath$. Since the \FOM has $3$ inputs and $3$ outputs, the frequency-response plot of it is composed of $3$ curves, i.e., corresponding to the $3$ singular values of $\bH(\imath \omega_j) \in \IC^{3 \times 3}$, for any frequency point $\omega_j$. The reduction order of all {\ROM}s is fixed at $r = 100$. 

The frequency response plots of the \FOM and the reduction errors related to four {\ROM}s are depicted in \Cref{fig:fig3}. We notice a good agreement between all responses and a comparable approximation quality for the intrusive and non-intrusive methods. Moreover, as to be expected, the \SPA and \QSPA methods lead to an improved fidelity in the low frequency range, whereas the errors for \BT and \QBT are specifically small for high frequencies.

\begin{figure}[tbh!]
\includegraphics[scale= 0.4]{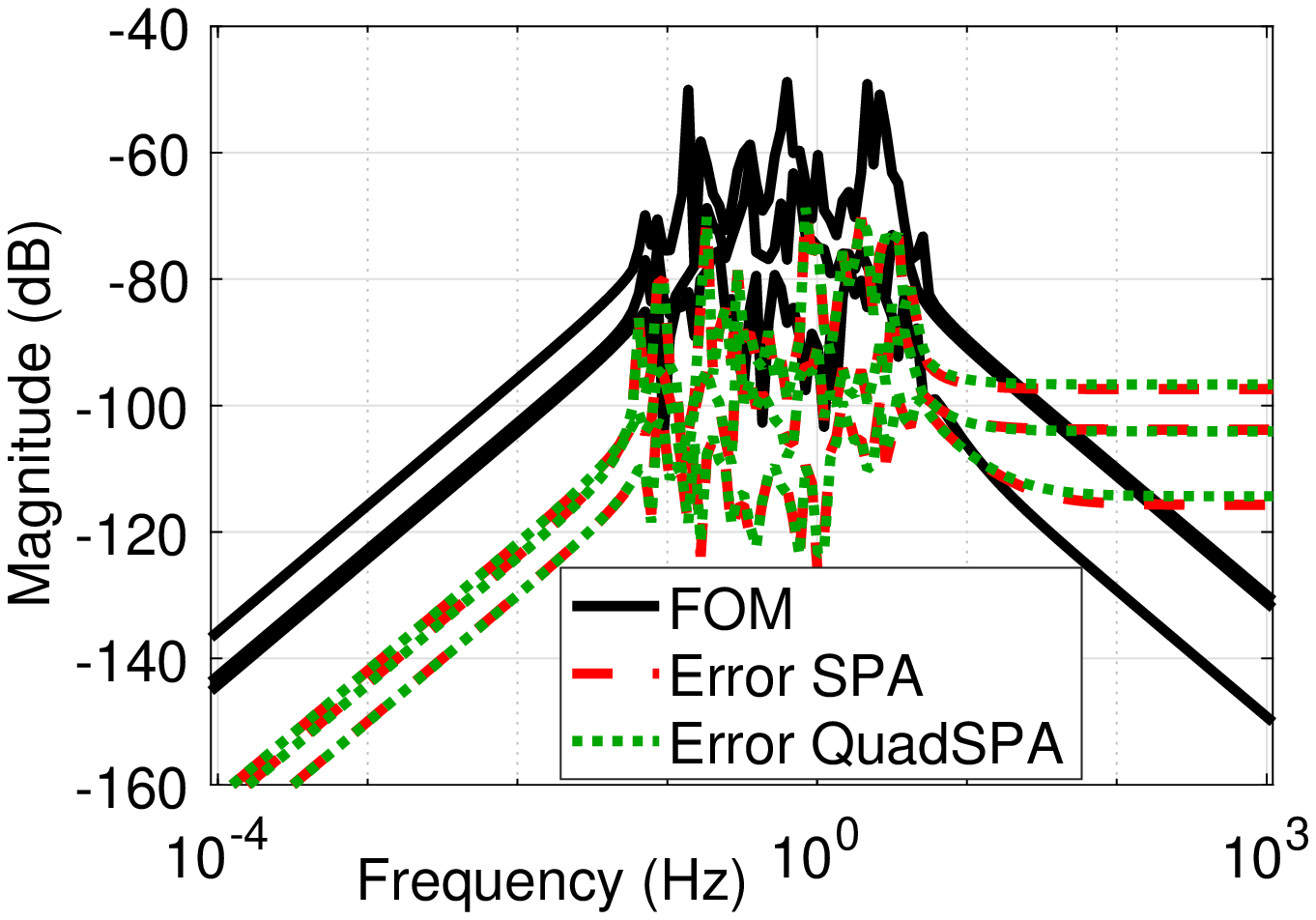} \hspace{1.5cm}
\includegraphics[scale= 0.4]{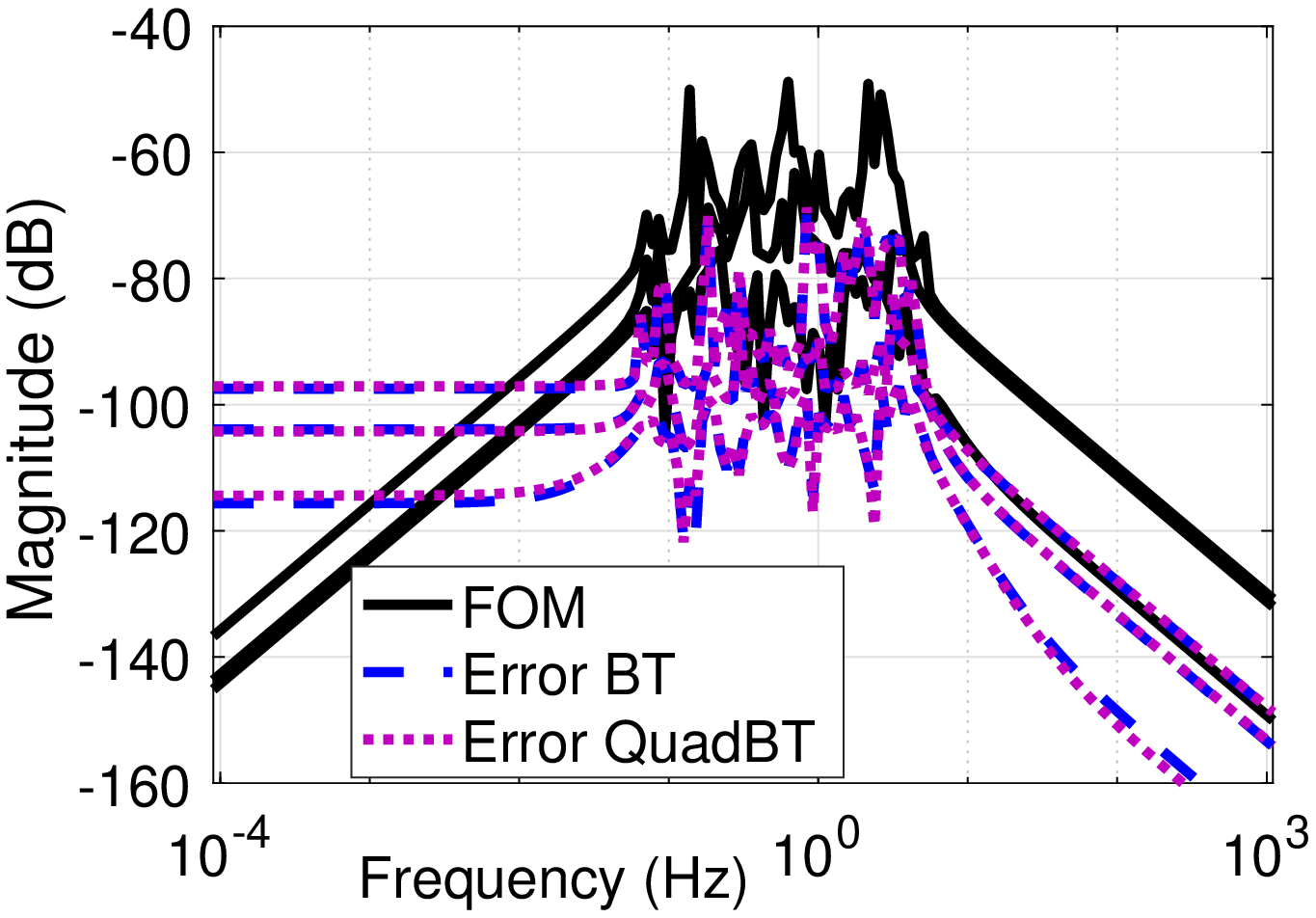}\\[0cm]
	\caption{\iss. Frequency response of \FOM and the errors using \SPA and \QSPA (left) and using \BT and \QBT; $\np = 1000$ and $r=100$ for all {\ROM}s.}
	\label{fig:fig3}\hspace{1.8cm}
\end{figure}

Then, in~\Cref{fig:fig4} we depict the configuration of the system poles for the \FOM, and for the two {\ROM}s based on quadrature approximations. The dominant poles of the original system seem to be well matched in both cases. Moreover, we would like to mention that the poles in the {\ROM}s occur in complex-conjugated pairs, which is in line with the realness of the fitted models.

\begin{SCfigure}
\raggedleft
\caption{\iss. Distribution of system poles for \QBT and \QSPA; $r = 100$ and $\np = 1000$ were used.\label{fig:fig4}}
\includegraphics[scale = 0.4]{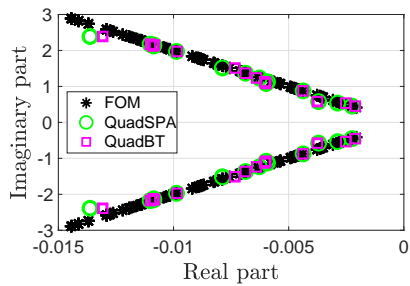}%
\end{SCfigure}

In the next experiment, we vary the number of points $\np$ between $600$ and $1\,000$, and we compute the relative approximation errors in the $\cH_{\infty}$ norm between the original system and the two {\ROM}s, i.e., computed by means of \QSPA and \QBT. As expected, the approximation error decreases as $\np$ increases, as shown in~\Cref{tab:table4}.

\begin{table}[h!]
	\centering
	\caption{\iss. Relative approximation errors in the $\cH_{\infty}$ norm for varying number $\np$ of quadrature points (\ROM dimension $r = 100$).}
	\label{tab:table4}
	\small
	\begin{tabular}{ |p{1.9cm}||p{1.7cm}|p{1.7cm}| p{1.7cm}| }
		\hline
		& $\np = 600$ & $\np = 800$  & $\np = 1000$ \\
		\hline
	\QSPA  & $5.6244 \cdot 10^{-1}$ & $ 1.6758 \cdot 10^{-1}$ & $3.2221 \cdot 10^{-2}$\\
		\hline
	\QBT  &  $6.0982 \cdot 10^{-1}$ & $4.4826 \cdot 10^{-2}$ & $ 3.2171 \cdot 10^{-2}$\\
		\hline
	\end{tabular}
		\small
\begin{tabular}{ |p{3.2cm}||p{3.2cm}|}
	\hline
	\SPA: $3.3153 \cdot 10^{-2}$ &   \BT: $3.2941 \cdot 10^{-2}$  \\
	\hline
\end{tabular}
\end{table}

Finally, we compute absolute deviations in the $\cH_\infty$ norm between the {\ROM}s computed with intrusive model reduction methods (\SPA and \BT), and their data-based counterparts. We vary $\np$ between $200$ and $1\,000$ and compute the $\cH_\infty$ quantities in~\Cref{tab:table5}. As the number of data points is increased, the deviation between the quadrature-based implementations \QSPA and \QBT and their intrusive implementation generally decreases and becomes negligible compared to the reduction error of the method.


\begin{table}[h!]
	\centering
	\caption{\iss. $\cH_{\infty}$ norm deviations between intrusive and quadrature-based methods for varying $\np$ (\ROM dimension is $r=100$.)}
	\label{tab:table5}
	\small
	\begin{tabular}{ |p{3.1cm}||p{1.7cm}|p{1.7cm}|p{1.7cm}|p{1.7cm}| p{1.7cm}| }
		\hline
		& $\np = 200$ & $\np = 600$ & $\np = 800$  & $\np = 1000$ \\
		\hline
		\SPA versus \QSPA   & $1.7813 \cdot 10^{-2}$ &  $6.6425 \cdot 10^{-3}$ & $1.9466 \cdot 10^{-3}$ & $2.6942 \cdot 10^{-4}$\\
		\hline
		\BT  versus \,\QBT  &  $2.0562 \cdot 10^{-2}$ &   $7.1800 \cdot 10^{-3}$ & $4.3901 \cdot 10^{-4}$ & $ 2.9588 \cdot 10^{-4}$\\
		\hline
	\end{tabular}
\end{table}

%% file: secConc.tex
\section{Conclusion} 
\label{sec:Conclusion}

This paper addressed shortcomings and issues corresponding to the practical usability of \SPA. Specifically, two new algorithmic implementations of this model order reduction method were proposed with different use cases in mind. The first one was a low-rank implementation that can be used for the reduction of systems much larger than the existing ones based on dense solvers. Notably, our \SPA implementation is of about the same computational complexity as the state-of-the-art algorithms for \BT.
Secondly, we derived a data-driven, non-intrusive reinterpretation of \SPA that is based on a quadrature approximation requiring solely input-output data. Since these data could also be obtained by measurements, the algorithm can be considered realization-free. It bears connections to the data-driven reinterpretation of \BT from \cite{gosea2022data} and thus behaves, in some sense, similarly. For example, the choice of data shows to be crucial for the approximation quality, similarly as it is also for the Loewner framework.
The findings of this paper were validated by several numerical tests, which illustrated the good correspondence between our new contributions for \SPA and the well-established results for \BT. Future research endeavors could include the study of generalized versions of \SPA with similar techniques as the ones used in this paper, both for the intrusive and the data-driven setting.

\section*{Acknowledgments}%
The authors would like to thank Peter Benner for providing insightful comments that helped to improve the manuscript. Moreover, the first author would like to thank for the support of the DFG research training group 2126 on algorithmic optimization.

%% file: AppQuad.tex
\section{Proof of \Cref{prop:quadMatrices}} \label{app:proofQBT}

For completeness, and since the proof of our main result \Cref{prop:quadSPA} follows using similar arguments, we state a proof for \Cref{prop:quadMatrices} in the following; cf.\,\cite{gosea2022data}.

Let the assumptions of the proposition hold, and let $\Kr: \IC \to \IC^{n,n}$, be given by $\Kr(s) = (s \bI - \bA)^{-1}$, so that $\spTF(s) = \bC \Kr(s) \bB$, as well as
\small
\begin{equation*}
\quadL^* = \left[ \begin{matrix}
\rho_1 \bC \Kr(\imunit \omega_1) \\
\vdots \\ \rho_\np \bC \Kr(\imunit \omega_\np)
\end{matrix}  \right] \quad  \text{and} \quad
\quadU = \left[ \,
\rho_1 \Kr(\imunit \omega_1)   \bB \ \ 
\cdots \quad \rho_\np  \Kr(\imunit \omega_\np) \bB  \right].
\end{equation*}
\normalsize
A direct calculation shows that
\small
\begin{align*}
\Kr(\imunit \omega_k) \Kr(\imunit \qpo_j) &= \frac{1}{\imunit \omega_k - \imunit \qpo_j} \Kr(\imunit \omega_k) \left[ (\imunit \omega_k \bI - \bA) -  ( \imunit \qpo_j\bI - \bA)  \right] \Kr( \imunit \qpo_j)\\
 &=  - \frac{1}{\imunit \omega_k -  \imunit \qpo_j} \left( \Kr(\imunit \omega_k)-\Kr(\imunit \qpo_j) \right).
\end{align*}
\normalsize
By the latter, we conclude
\small
\begin{align*}
	\quadLL_{k,j} &= \quadL_k \quadU_j = \rho_k \qwe_j\bC \Kr(\imunit \omega_k) \Kr(\imunit \qpo_j) \bB = - \frac{\rho_k \qwe_j}{\imunit \omega_k -  \imunit \qpo_j} \left( \bC \Kr(\imunit \omega_k) \bB- \bC \Kr(\imunit \qpo_j) \bB \right) \\
	&= - \rho_k \qwe_j	\displaystyle \frac{\spTF(\imunit\omega_k) - \spTF(\imunit\qpo_j)}{\imunit\omega_k - \imunit \qpo_j}.
\end{align*}
\normalsize
Similarly, by using
\small
\begin{align*}
\Kr(\imunit \omega_k) \bA \Kr(\imunit \qpo_j) &= \frac{1}{\imunit \omega_k - \imunit \qpo_j} \Kr(\imunit \omega_k) \left[ \imunit \qpo_j(\imunit \omega_k \bI - \bA) -  \imunit \omega_k ( \imunit \qpo_j\bI - \bA)  \right] \Kr( \imunit \qpo_j)\\
&=  - \frac{1}{\imunit \omega_k -  \imunit \qpo_j} \left( \imunit \omega_k \Kr(\imunit \omega_k)-\imunit \qpo_j\Kr(\imunit \qpo_j) \right),
\end{align*}
\normalsize
we can write
\small
\begin{align*}
\quadMM_{k,j} &= 
\rho_k \qwe_j\bC \Kr(\imunit \omega_k) \bA \Kr(\imunit \qpo_j) \bB = - \frac{\rho_k \qwe_j}{\imunit \omega_k -  \imunit \qpo_j} \left( \imunit \omega_k \bC \Kr(\imunit \omega_k) \bB - \imunit \qpo_j\bC \Kr(\imunit \qpo_j) \bB \right) \nonumber \\
&= - \rho_k \qwe_j	\displaystyle \frac{\imunit\omega_k \spTF(\imunit\omega_k) - \imunit\qpo_j\spTF(\imunit\qpo_j)}{\imunit\omega_k - \imunit \qpo_j}. 
\end{align*}
\normalsize
Finally, the representations of $\quadLb$ and $\quadcU$ (as data matrices) claimed in the proposition, also follows in a straightforward manner. \hfill $\square$
\section{An extension of \Cref{prop:quadSPA}} \label{app:QuadBalancing}
The assumption $\omega_k \neq \qpo_j$ for $k,j \in \{1,\ldots,\nq\}$ can be omitted in \Cref{prop:quadSPA}. However, in that cases, samples corresponding to the derivative of the transfer function are required.

The required adaptions concern the data matrices $\requadLL$ and $\requadMM$ only, given as
\small
\begin{align*} 
	\requadLL_{k,j} &= 
	\begin{cases} - \rho_k \qwe_j 
	\displaystyle \frac{\TFzer(\imunit\omega_k) - \TFzer(\imunit\qpo_j)}{\imunit\omega_k -  \imunit \qpo_j} &  \omega_k \neq \qpo_j	\\
     -\rho_k \qwe_j  \frac{d}{ds}\TFzer(s)_{s=\imunit\qpo_j} &  \omega_k = \qpo_j, 
	\end{cases}
	\\
	\requadMM_{k,j} &= 
	\begin{cases} 
	- \rho_k \qwe_j \displaystyle \frac{(\imunit\omega_k)^{-1} \TFzer(\imunit\omega_k) - (\imunit\qpo_j)^{-1} \TFzer(\imunit\qpo_j)}{\imunit\omega_k -  \imunit \qpo_j} & \omega_k \neq \qpo_j	 \\
	 -\rho_k \qwe_j \left( (\imunit \qpo_j)^{-1} \frac{d}{ds}\TFzer(s)_{s=\imunit\qpo_j} - (\imunit \qpo_j)^{-2}\TFzer(\imunit\qpo_j) \right)  &  \omega_k =\qpo_j	.
	\end{cases}
	\end{align*}
	\normalsize
In the light of the proof that was presented for \Cref{prop:quadSPA}, it only remains to show the validity of the representations $\requadLL_{k,j}$ and $\requadMM_{k,j}$ for the special case $\omega_k = \qpo_j$. We do this considering the limit $\omega_k \to \qpo_j$ of the representations that have already been shown for $\omega_k \neq \qpo_j$ using L'Hospital's rule.

For $\requadMM_{k,j}$ with $\omega_k = \qpo_j$, the crucial step is
\small
\begin{align*}
	\lim_{\omega_k \to \qpo_j}  \frac{\TFzer(\imunit\omega_k) - \TFzer(\imunit\qpo_j)}{\imunit\omega_k -  \imunit \qpo_j}
	 = \lim_{\omega_k \to \qpo_j}  \frac{ \frac{d}{ds}\left( \TFzer(s) - \TFzer(\imunit\qpo_j)\right)_{s = \imunit \omega_k} }{\frac{d}{ds} \left(s -\imunit \qpo_j \right)_{s= \imunit \omega_k}} = \frac{d}{ds}\TFzer(s)_{s=\imunit\qpo_j}.
\end{align*}
\normalsize
From the latter result, the claimed representation directly follows. Similarly, the expression for $\requadLL_{k,j}$ and $\omega_k = \qpo_j$ can be shown using
\small
\begin{align*}
	\lim_{\omega_k \to \qpo_j}   \frac{(\imunit\omega_k)^{-1} \TFzer(\imunit\omega_k) - (\imunit\qpo_j)^{-1} \TFzer(\imunit\qpo_j)}{\imunit\omega_k -  \imunit \qpo_j}
	  =  \frac{d}{ds}\left(s^{-1} \TFzer(s) \right)_{s=  \imunit \qpo_j} \\
	  =  (\imunit \qpo_j)^{-1} \frac{d}{ds}\TFzer(s)_{s=\imunit\qpo_j} - (\imunit \qpo_j)^{-2}\TFzer(\imunit\qpo_j) .
\end{align*}
\normalsize



%% file: LiGoSPAarXiv.bbl
\begin{thebibliography}{10}

\bibitem{ACA05}
A.~C. Antoulas.
\newblock {\em Approximation of Large-Scale Dynamical Systems}, volume~6 of
  {\em Adv. Des. Control}.
\newblock {SIAM} Publications, Philadelphia, PA, 2005.
\newblock \href {https://doi.org/10.1137/1.9780898718713}
  {\path{doi:10.1137/1.9780898718713}}.

\bibitem{AntBG20}
A.~C. Antoulas, C.~A. Beattie, and S.~Gugercin.
\newblock {\em Interpolatory Methods for Model Reduction}.
\newblock Computational Science \& Engineering. SIAM, Philadelphia, PA, 2020.
\newblock \href {https://doi.org/10.1137/1.9781611976083}
  {\path{doi:10.1137/1.9781611976083}}.

\bibitem{baur2008gramian}
U.~Baur and P.~Benner.
\newblock Gramian-based model reduction for data-sparse systems.
\newblock {\em {SIAM} J. Sci. Comput.}, 31(1):776--798, 2008.
\newblock \href {https://doi.org/10.1137/070711578}
  {\path{doi:10.1137/070711578}}.

\bibitem{benner2021model}
P.~Benner and L.~Feng.
\newblock Model order reduction based on moment-matching.
\newblock In {\em Volume 1: System-and Data-Driven Methods and Algorithms,
  Chapter 3}, pages 57--96. De Gruyter, 2021.
\newblock \href {https://doi.org/10.1515/9783110498967-003}
  {\path{doi:10.1515/9783110498967-003}}.

\bibitem{morBenGKetal20}
P.~Benner, P.~Goyal, B.~Kramer, B.~Peherstorfer, and K.~Willcox.
\newblock Operator inference for non-intrusive model reduction of systems with
  non-polynomial nonlinear terms.
\newblock {\em Comp. Meth. Appl. Mech. Eng.}, 372:113433, 2020.
\newblock \href {https://doi.org/10.1016/j.cma.2020.113433}
  {\path{doi:10.1016/j.cma.2020.113433}}.

\bibitem{HandbookVol1}
P.~Benner, S.~Grivet-Talocia, A.~Quarteroni, G.~Rozza, W.~H.~A. Schilders, and
  L.~M. Silveira, editors.
\newblock {\em {Model Order Reduction. Volume 1: System- and Data-Driven
  Methods and Algorithms}}.
\newblock De~Gruyter, Berlin, 2021.
\newblock \href {https://doi.org/10.1515/9783110498967}
  {\path{doi:10.1515/9783110498967}}.

\bibitem{morBenGQetal21b}
P.~Benner, S.~Grivet-Talocia, A.~Quarteroni, G.~Rozza, W.~H.~A. Schilders, and
  L.~M. Silveira, editors.
\newblock {\em {Model Order Reduction. Volume 3: Applications}}.
\newblock De~Gruyter, Berlin, 2021.
\newblock \href {https://doi.org/10.1515/9783110499001}
  {\path{doi:10.1515/9783110499001}}.

\bibitem{benner2015survey}
P.~Benner, S.~Gugercin, and K.~Willcox.
\newblock A survey of projection-based model reduction methods for parametric
  dynamical systems.
\newblock {\em {SIAM} Rev.}, 57(4):483--531, 2015.
\newblock \href {https://doi.org/10.1137/130932715}
  {\path{doi:10.1137/130932715}}.

\bibitem{benner2013efficient}
P.~Benner, P.~K{\"u}rschner, and J.~Saak.
\newblock Efficient handling of complex shift parameters in the low-rank
  {C}holesky factor {ADI} method.
\newblock {\em Numerical Algorithms}, 62:225--251, 2013.
\newblock \href {https://doi.org/10.1007/s11075-012-9569-7}
  {\path{doi:10.1007/s11075-012-9569-7}}.

\bibitem{book:dimred2003}
P.~Benner, V.~Mehrmann, and D.~C. Sorensen.
\newblock {\em Dimension Reduction of Large-Scale Systems}, volume~45 of {\em
  Lect. Notes Comput. Sci. Eng.}
\newblock Springer-Verlag, Berlin/Heidelberg, Germany, 2005.
\newblock \href {https://doi.org/10.1007/3-540-27909-1}
  {\path{doi:10.1007/3-540-27909-1}}.

\bibitem{BOCW17}
P.~Benner, M.~Ohlberger, A.~Cohen, and K.~Willcox, editors.
\newblock {\em Model Reduction and Approximation: Theory and Algorithms}.
\newblock Computational Science \& Engineering. SIAM, Philadelphia, PA, 2017.
\newblock \href {https://doi.org/10.1137/1.9781611974829}
  {\path{doi:10.1137/1.9781611974829}}.

\bibitem{morBenQQ00}
P.~Benner, E.~S. Quintana-Ort{\'\i}, and G.~Quintana-Ort{\'\i}.
\newblock Balanced truncation model reduction of large-scale dense systems on
  parallel computers.
\newblock {\em Math. Comput. Model. Dyn. Syst.}, 6(4):383--405, 2000.
\newblock \href {https://doi.org/10.1076/mcmd.6.4.383.3658}
  {\path{doi:10.1076/mcmd.6.4.383.3658}}.

\bibitem{benner2000singular}
P.~Benner, E.~S. Quintana-Orti, and G.~Quintana-Ort{\'\i}.
\newblock Singular perturbation approximation of large, dense linear systems.
\newblock In {\em CACSD. Conference Proceedings. IEEE International Symposium
  on Computer-Aided Control System Design (Cat. No. 00TH8537)}, pages 255--260.
  IEEE, 2000.
\newblock \href {https://doi.org/10.1109/CACSD.2000.900220}
  {\path{doi:10.1109/CACSD.2000.900220}}.

\bibitem{benner2010balanced}
P.~Benner and A.~Schneider.
\newblock {Balanced Truncation Model Order Reduction for {LTI} Systems with
  many Inputs or Outputs}.
\newblock In Andr{\'a}s Edelmayer, editor, {\em Proc. of the 19th International
  Symposium on Mathematical Theory of Networks and Systems}, pages 1971--1974,
  Budapest, Hungary, 2010.
\newblock isbn: 978-963-311-370-7.

\bibitem{BERTRAM202166}
C.~Bertram and H.~Fassbender.
\newblock A quadrature framework for solving {L}yapunov and {S}ylvester
  equations.
\newblock {\em Linear Algebra Appl.}, 622:66--103, 2021.
\newblock \href {https://doi.org/10.1016/j.laa.2021.03.029}
  {\path{doi:10.1016/j.laa.2021.03.029}}.

\bibitem{binder2022texttt}
A.~Binder, V.~Mehrmann, A.~Miedlar, and P.~Schulze.
\newblock A {M}atlab toolbox for the regularization of descriptor systems
  arising from generalized realization procedures.
\newblock e-print 2212.02604, arXiv, 2022.
\newblock URL: \url{https://arxiv.org/pdf/2212.02604.pdf}.

\bibitem{breiten20212}
T.~Breiten and T.~Stykel.
\newblock Balancing-related model reduction methods.
\newblock {\em Volume 1: System-and Data-Driven Methods and Algorithms, Chapter
  2}, pages 15--56, 2021.
\newblock \href {https://doi.org/10.1515/9783110498967-002}
  {\path{doi:10.1515/9783110498967-002}}.

\bibitem{morChaV02}
Y.~Chahlaoui and P.~Van~Dooren.
\newblock A collection of benchmark examples for model reduction of linear time
  invariant dynamical systems.
\newblock Technical Report 2002--2, SLICOT Working Note, 2002.
\newblock Available from \url{www.slicot.org}.

\bibitem{Drmac2022}
Z.~Drma{\v{c}} and B.~Peherstorfer.
\newblock {\em Learning Low-Dimensional Dynamical-System Models from Noisy
  Frequency-Response Data with Loewner Rational Interpolation}, pages 39--57.
\newblock Springer International Publishing, Cham, 2022.
\newblock \href {https://doi.org/10.1007/978-3-030-95157-3_3}
  {\path{doi:10.1007/978-3-030-95157-3_3}}.

\bibitem{fernando1982singular}
K.~V. Fernando and H.~Nicholson.
\newblock Singular perturbational model reduction of balanced systems.
\newblock {\em {IEEE} Trans. Autom. Control}, 27(2):466--468, 1982.
\newblock \href {https://doi.org/10.1109/TAC.1982.1102932}
  {\path{doi:10.1109/TAC.1982.1102932}}.

\bibitem{glover1984all}
K.~Glover.
\newblock All optimal {H}ankel-norm approximations of linear multivariable
  systems and their {L}$^\infty$-error norms.
\newblock {\em Internat. J. Control}, 39(6):1115--1193, 1984.
\newblock \href {https://doi.org/10.1080/00207178408933239}
  {\path{doi:10.1080/00207178408933239}}.

\bibitem{gosea2022data}
I.~V. Gosea, S.~Gugercin, and C.~Beattie.
\newblock Data-driven balancing of linear dynamical systems.
\newblock {\em {SIAM} J. Sci. Comput.}, 44(1):A554--A582, 2022.
\newblock \href {https://doi.org/10.1137/21M1411081}
  {\path{doi:10.1137/21M1411081}}.

\bibitem{book:green2012linear}
M.~Green and D.~J.~N. Limebeer.
\newblock {\em Linear robust control}.
\newblock Prentice-Hall, Inc., 1994.
\newblock \href {https://doi.org/10.5555/191301} {\path{doi:10.5555/191301}}.

\bibitem{morGugAB01}
S.~Gugercin, A.~C. Antoulas, and M.~Bedrossian.
\newblock Approximation of the international space station 1{R} and 12{A} flex
  models.
\newblock In {\em Proceedings of the IEEE Conf. Decision Contr.}, pages
  1515--1516, 2001.
\newblock \href {https://doi.org/10.1109/CDC.2001.981109}
  {\path{doi:10.1109/CDC.2001.981109}}.

\bibitem{morGugL05}
S.~Gugercin and J.-R. Li.
\newblock Smith-type methods for balanced truncation of large systems.
\newblock In P.~Benner, V.~Mehrmann, and D.~Sorensen, editors, {\em Dimension
  Reduction of Large-Scale Systems}, volume~45 of {\em Lect. Notes Comput. Sci.
  Eng.}, pages 49--82. Springer-Verlag, Berlin/Heidelberg, Germany, 2005.
\newblock \href {https://doi.org/10.1007/3-540-27909-1_2}
  {\path{doi:10.1007/3-540-27909-1_2}}.

\bibitem{guiver2018generalised}
C.~Guiver.
\newblock The generalised singular perturbation approximation for bounded real
  and positive real control systems.
\newblock {\em Mathematical Control and Related Fields}, 9(2), 2018.
\newblock \href {https://doi.org/10.1109/TCS.1976.1084254}
  {\path{doi:10.1109/TCS.1976.1084254}}.

\bibitem{li2002low}
J.~R. Li and J.~White.
\newblock Low rank solution of {L}yapunov equations.
\newblock {\em {SIAM} J. Matrix Anal. Appl.}, 24(1):260--280, 2002.
\newblock \href {https://doi.org/10.1137/S0895479801384937}
  {\path{doi:10.1137/S0895479801384937}}.

\bibitem{code:LiGoSPA}
B.~Liljegren-Sailer and I.~V. Gosea.
\newblock Code for paper '{D}ata-driven and low-rank implementations of
  balanced singular perturbation approximation'.
\newblock \url{https://doi.org/10.5281/zenodo.7671264}, 2023.

\bibitem{liu1989singular}
Y.~Liu and B.~D.~O. Anderson.
\newblock Singular perturbation approximation of balanced systems.
\newblock {\em Internat. J. Control}, 50(4):1379--1405, 1989.
\newblock \href {https://doi.org/10.1080/00207178908953437}
  {\path{doi:10.1080/00207178908953437}}.

\bibitem{ajm07}
A.~J. Mayo and A.~C. Antoulas.
\newblock A framework for the solution of the generalized realization problem.
\newblock {\em Linear Algebra Appl.}, 425(2--3):634--662, 2007.
\newblock Special Issue in honor of P.~A. Fuhrmann, Edited by A.~C. Antoulas,
  U. Helmke, J. Rosenthal, V. Vinnikov, and E. Zerz.
\newblock \href {https://doi.org/10.1016/j.laa.2007.03.008}
  {\path{doi:10.1016/j.laa.2007.03.008}}.

\bibitem{moore1981principal}
B.~C. Moore.
\newblock Principal component analysis in linear systems: controllability,
  observability, and model reduction.
\newblock {\em {IEEE} Trans. Autom. Control}, AC--26(1):17--32, 1981.
\newblock \href {https://doi.org/10.1109/TAC.1981.1102568}
  {\path{doi:10.1109/TAC.1981.1102568}}.

\bibitem{mullis1976synthesis}
C.~Mullis and R.~Roberts.
\newblock Synthesis of minimum roundoff noise fixed point digital filters.
\newblock {\em {IEEE} Trans. Circuits Syst.}, 23(9):551--562, 1976.
\newblock \href {https://doi.org/10.1109/TCS.1976.1084254}
  {\path{doi:10.1109/TCS.1976.1084254}}.

\bibitem{opmeer2011model}
M.~R. Opmeer.
\newblock Model order reduction by balanced proper orthogonal decomposition and
  by rational interpolation.
\newblock {\em {IEEE} Trans. Autom. Control}, 57(2):472--477, 2012.
\newblock \href {https://doi.org/10.1109/TAC.2011.2164018}
  {\path{doi:10.1109/TAC.2011.2164018}}.

\bibitem{PEHERSTORFER2016196}
B.~Peherstorfer and K.~Willcox.
\newblock Data-driven operator inference for nonintrusive projection-based
  model reduction.
\newblock {\em Comp. Meth. Appl. Mech. Eng.}, 306:196--215, 2016.
\newblock \href {https://doi.org/10.1016/j.cma.2016.03.025}
  {\path{doi:10.1016/j.cma.2016.03.025}}.

\bibitem{penzl1999cyclic}
T.~Penzl.
\newblock A cyclic low-rank {S}mith method for large sparse {L}yapunov
  equations.
\newblock {\em {SIAM} J. Sci. Comput.}, 21(4):1401--1418, 1999.
\newblock \href {https://doi.org/10.1137/S1064827598347666}
  {\path{doi:10.1137/S1064827598347666}}.

\bibitem{quarteroni2015reduced}
A.~Quarteroni, A.~Manzoni, and F.~Negri.
\newblock {\em Reduced Basis Methods for Partial Differential Equations},
  volume~92 of {\em La Matematica per il 3+2}.
\newblock Springer International Publishing, 2016.

\bibitem{Ro05}
C.~W. Rowley.
\newblock Model reduction for fluids, using balanced proper orthogonal
  decomposition.
\newblock {\em Int. J. Bifurcat. Chaos}, 15(3):997--1013, 2005.
\newblock \href {https://doi.org/10.1142/S0218127405012429}
  {\path{doi:10.1142/S0218127405012429}}.

\bibitem{SaaKB-mmess}
J.~Saak, M.~K\"{o}hler, and P.~Benner.
\newblock {M-M.E.S.S.-2.1} -- {T}he {M}atrix {E}quations {S}parse {S}olvers
  library.
\newblock \url{https://doi.org/10.5281/zenodo.4719688}, 2021.

\bibitem{schmid_2010}
P.~J. Schmid.
\newblock Dynamic mode decomposition of numerical and experimental data.
\newblock {\em J. of Fluid Mech.}, 656:5–28, 2010.
\newblock \href {https://doi.org/10.1017/S0022112010001217}
  {\path{doi:10.1017/S0022112010001217}}.

\bibitem{ValSim16}
V.~Simoncini.
\newblock Computational methods for linear matrix equations.
\newblock {\em {SIAM} Rev.}, 58(3):377--441, 2016.
\newblock \href {https://doi.org/10.1137/130912839}
  {\path{doi:10.1137/130912839}}.

\bibitem{sorensen2003direct}
D.~C. Sorensen and Y.~Zhou.
\newblock Direct methods for matrix {S}ylvester and {L}yapunov equations.
\newblock {\em J. Appl. Math.}, 2003(6):277--303, 2003.
\newblock \href {https://doi.org/10.1155/S1110757X03212055}
  {\path{doi:10.1155/S1110757X03212055}}.

\bibitem{Varga91}
A.~Varga.
\newblock Balancing free square-root algorithm for computing singular
  perturbation approximations.
\newblock In {\em Proc. of the 30th IEEE Conf. Decision Contr.}, pages
  1062--1065, 1991.
\newblock \href {https://doi.org/10.1109/CDC.1991.261486}
  {\path{doi:10.1109/CDC.1991.261486}}.

\bibitem{willcox2002balanced}
K.~Willcox and J.~Peraire.
\newblock Balanced model reduction via the proper orthogonal decomposition.
\newblock {\em AIAA J.}, 40(11):2323--2330, 2002.
\newblock \href {https://doi.org/10.2514/2.1570} {\path{doi:10.2514/2.1570}}.

\bibitem{zhang2021factorization}
Q.~Zhang, I.~V. Gosea, and A.~C. Antoulas.
\newblock Factorization of the {L}oewner matrix pencil and its consequences.
\newblock e-print 2103.09674, arXiv, 2021.
\newblock cs.NA.
\newblock URL: \url{https://arxiv.org/abs/2103.09674}.

\end{thebibliography}
